\documentclass[11pt]{amsart}
\usepackage{amssymb}
\usepackage{amsmath}
\usepackage{amsfonts}
\usepackage{graphicx}

\usepackage{hyperref}
    \usepackage{aeguill}
    \usepackage{type1cm}

\theoremstyle{plain}
\newtheorem{thm}{Theorem}[section]

\newtheorem{claim}[thm]{Claim}

\newtheorem{corollary}[thm]{Corollary}

\newtheorem{example}[thm]{Example}

\newtheorem{lemma}[thm]{Lemma}

\newtheorem{proposition}[thm]{Proposition}
\newtheorem{remark}[thm]{Remark}

\newtheorem{theorem}[thm]{Theorem}
\numberwithin{equation}{section}
\newcommand{\N}{\mathbb{N}}

\newcommand{\R}{\mathbb{R}}

\newcommand{\Rn}{\mathbb{R}^n}

\DeclareMathOperator{\sop}{supp\,\!}

\DeclareMathOperator{\diam}{diam\,\!}

\usepackage[usenames,dvipsnames]{color}

\usepackage[dvipsnames]{xcolor}

\begin{document}

\title[$C^1$ and $C^{1,1}$ convex extensions of convex functions]{Whitney Extension Theorems for convex functions of the classes $C^1$ and $C^{1,\omega}$.}
\author{Daniel Azagra}
\address{ICMAT (CSIC-UAM-UC3-UCM), Departamento de An{\'a}lisis Matem{\'a}tico,
Facultad Ciencias Matem{\'a}ticas, Universidad Complutense, 28040, Madrid, Spain }
\email{azagra@mat.ucm.es}

\author{Carlos Mudarra}
\address{ICMAT (CSIC-UAM-UC3-UCM), Calle Nicol\'as Cabrera 13-15.
28049 Madrid SPAIN}
\email{carlos.mudarra@icmat.es}

\date{September 8, 2016}


\thanks{D. Azagra was partially supported by Ministerio de Educaci\'on, Cultura y Deporte, Programa Estatal de Promoci\'on del Talento y su Empleabilidad en I+D+i, Subprograma Estatal de Movilidad. C. Mudarra was supported by Programa Internacional de Doctorado Fundaci\'on La Caixa--Severo Ochoa. Both authors partially supported by MTM2012-34341.}


\begin{abstract}
Let $C$ be a subset of $\mathbb{R}^n$ (not necessarily convex), $f:C\to\mathbb{R}$ be a function, 
and $G:C\to\mathbb{R}^n$ be a uniformly continuous function, with modulus of continuity $\omega$. We provide a necessary and sufficient condition on $f$, $G$ for the existence of a {\em convex} function $F\in C^{1, \omega}(\mathbb{R}^n)$ such that $F=f$ on $C$ and $\nabla F=G$ on $C$, with a good control of the modulus of continuity of $\nabla F$ in terms of that of $G$. On the other hand, assuming that $C$ is compact, we also solve a similar problem for the class of $C^1$ convex functions on $\mathbb{R}^n$, with a good control of the Lipschitz constants of the extensions (namely, $\textrm{Lip}(F)\lesssim \|G\|_{\infty}$). Finally, we give a geometrical application concerning interpolation of compact subsets $K$ of $\mathbb{R}^n$ by boundaries of $C^1$ or $C^{1,1}$ convex bodies with prescribed outer normals on $K$.
\end{abstract}

\maketitle

\section{Introduction and main results}

Throughout this paper, by a modulus of continuity $\omega$ we understand a concave, strictly increasing function $\omega:[0,\infty)\to [0, \infty)$ such that $\omega(0^{+})=0$. In particular, $\omega$ has an inverse $\omega^{-1}:[0,\beta)\to [0, \infty)$ which is convex and strictly increasing, where $\beta>0$ may be finite or infinite (according to whether $\omega$ is bounded or unbounded). Furthermore, $\omega$ is subadditive, and satisfies $\omega(\lambda t)\leq \lambda \omega (t)$ for $\lambda\geq 1$, and $\omega(\mu t)\geq \mu \omega(t)$ for $0\leq \mu\leq 1$. This in turn implies that $\omega^{-1}(\mu s)\leq \mu \omega^{-1}(s)$ for $\mu\in [0,1]$, and $\omega^{-1}(\lambda s)\geq \lambda \omega^{-1}(s)$ for $\lambda\geq 1$. It is well-known that for every uniformly continuous function $f:X\to Y$ between two metric spaces there exists a modulus of continuity $\omega$ such that $d_{Y}(f(x), f(z))\leq \omega\left(d_{X}(x,z)\right)$ for every $x,z\in X$. Slightly abusing terminology, we will say that a mapping $G:X\to Y$ has modulus of continuity $\omega$ (or that $G$ is $\omega$-continuous) if there exists $M\geq 0$ such that
$$
d_{Y}\left(G(x),G(y)\right)\leq M\omega\left(d_{X}(x,y)\right)
$$
for all $x, y\in X$.

Let $C$ be a subset of $\Rn$, $f:C\to\R$ a function, and $G:C\to\Rn$ a uniformly continuous mapping with modulus of continuity $\omega$. 
Assume also  that the pair $(f,G)$ satisfies
$$
|f(x)-f(y)-\langle G(y), x-y\rangle|\leq M|x-y|\,\omega\left(|x-y|\right) \eqno(W^{1, \omega})
$$
for every $x, y\in C$ (in the case that $\omega(t)=t$ we will also denote this condition by $(W^{1,1})$). Then a well-known version of the Whitney extension theorem for the class $C^{1,\omega}$ due to Glaeser holds true (see \cite{Glaeser, Stein}), and we get the existence of a function $F\in C^{1, \omega}(\Rn)$ such that $F=f$ and $\nabla F=G$ on $C$. 

It is natural to ask what further assumptions on $f$, $G$ would be necessary and sufficient to ensure that $F$ can be taken to be convex.
In a recent paper \cite{AM}, we solved a similar problem for the class of $C^{\infty}$, under the much more stringent assumptions that $C$ be convex and compact. We refer to the introduction of \cite{AM} and the references therein for background; see in particular \cite{FeffermanSurvey, BrudnyiBrudnyi} for an account of the spectacular progress made on Whitney extension problems in the last decade. Here we will only mention that results of this nature for the special class of convex functions have interesting applications in differential geometry, PDE theory (such as Monge-Amp\`ere equations), nonlinear dynamics, and quantum computing, see \cite{GhomiJDG2001, GhomiPAMS2002, GhomiBLMS2004, MinYan} and the references therein. 
We should also note here that, in contrast with the classical Whitney extension theorem \cite{Whitney} (concerning jets) and also with the solutions \cite{Glaeser, BrudnyiShvartsman, Fefferman2005, Fefferman2006} to the Whitney extension problem (concerning functions), in whose proofs one can use appropriate partitions of unity in order to patch local solutions together to obtain a global solution, such tools are no longer available in our setting. Moreover, further difficulties arise from the rigid global behaviour of convex functions, see Proposition \ref{rigid global behaviour of convex functions} for instance. The following example illustrates both of these issues: take any four numbers $a, b, c, d\in\R$ with $a<b<0<c<d$, and define $C=\{a, b, 0, c, d\}$ and $f(x)=|x|$ for $x\in C$. Since $C$ is a five-point set it is clear that there are infinitely many $C^{1}$ functions (even infinitely many polynomials) $F$ with $F=f$ on $C$. However, none of these $F$ can be convex on $\R$, because, as is easily checked, any convex extension $g$ of $f$ to $\R$ must satisfy $g(x)=|x|$ for every $x\in [a, d]$, and therefore $g$ cannot be differentiable at $0$.
On the other hand it should be noted that, following the works of Brudnyi-Shvartsman \cite{BrudnyiShvartsman} for the solution of the $C^{1,1}$ Whitney extension problem, and of Fefferman \cite{Fefferman2005, Fefferman2006} concerning the solutions of the $C^{m-1,1}$ and $C^{m}$ Whitney extension problems in full generality, there is a natural interpretation of the term {\em local condition} in extension theorems that refers to the existence of a {\em finiteness principle}, which states that extendibility (with controlled norm) of a function from all finite subsets of $C$ of cardinality at most $k$ (for some $k<\infty$ fixed) implies extendibility (with controlled norm) of the function defined on all of $C$. It would be very interesting to know whether such a finiteness principle holds for $C^{1,1}$ convex extension of functions. Thus the following disclaimer is in order: whenever we refer to global versus local conditions in this paper we are merely talking about difficulties like those we have mentioned, which do not contradict the existence of a finiteness principle for $C^{1,1}$ convex extension.

Let us now introduce one global condition which, we have found, is necessary and sufficient for a function $f:C\to\R$ (and a mapping $G:C\to\R^n$ with modulus of continuity $\omega$) to have a convex extension $F$ of class $C^{1, \omega}(\Rn)$ such that $\nabla F=G$ on $C$.
For a mapping $G:C\subset\Rn\to\Rn$ we will denote
\begin{equation}\label{definition of M(G,C)}
M(G, C):=\sup_{x, y\in C, \, x\neq y}\frac{|G(x)-G(y)|}{\omega(|x-y|)}.
\end{equation}
We will say that $f$ and $G$ satisfy the property $(CW^{1,\omega})$ on $C$ if either $G$ is constant, or else $0<M(G, C)<\infty$ and there exists a constant $\eta\in (0,1/2]$ such that
$$
f(x)-f(y)-\langle G(y), x-y\rangle \geq \eta \, |G(x)-G(y)| \, \omega^{-1}\left( \frac{1}{2M} |G(x)-G(y)| \right)
$$
$$
\textrm{ for all } x, y\in C, \textrm{ where } M=M(G, C). \eqno(CW^{1,\omega})
$$
Throughout this paper we will assume that $0<M<\infty$, or equivalently that $G$ is nonconstant and has modulus of continuity $\omega$. We may of course do so, because if $M=0$ then our problem has a trivial solution (namely, the function $x\mapsto f(x_0)+ \langle G(x_0), x-x_0\rangle$ defines an affine extension of $f$ to $\Rn$, for any $x_0\in C$).

In the case that $\omega(t)=t$, we will also denote this condition by $(CW^{1,1})$. That is, $(f,G)$ satisfies $(CW^{1,1})$ if and only if there exists $\delta>0$ such that
$$
f(x)-f(y)-\langle G (y), x-y \rangle \geq \delta \: | G (x)- G(y) |^{2},
\eqno(CW^{1,1})
$$ 
for all $x,y\in C$.
\begin{remark}\label{CW11 implies W11} \hfill{ }
\begin{enumerate} 
\item
If $(f,G)$ satisfies condition $(CW^{1, \omega})$ and $G$ has modulus of continuity $\omega$, then $(f,G)$ satisfies condition $(W^{1, \omega})$.
\item If for some $M>0$ the pair $(f,G)$ satisfies 
$$
f(x)-f(y)-\langle G(y), x-y\rangle \geq \frac{1}{2} \, |G(x)-G(y)| \, \omega^{-1}\left( \frac{1}{2M} |G(x)-G(y)| \right)
$$
for all $x, y\in C$, then $G$ is $\omega$-continuous and $$
\sup_{x,y\in C, x\neq y}\frac{|G(x)-G(y)|}{\omega(|x-y|)}\leq 2M.
$$
\end{enumerate}
\end{remark}
\begin{proof}
Condition $(CW^{1, \omega})$ implies that
$$
0 \leq f(x)-f(y)-\langle G(y), x-y \rangle \leq \langle G(y)-G(x), y-x \rangle , 
$$
for all $x,y \in C$ with $x\neq y$,
hence, 
$$
0 \leq \frac{f(x)-f(y)-\langle G(y), x-y \rangle}{|x-y| \omega(|x-y|)} \leq \frac{\langle G(y)-G(x), y-x \rangle}{|x-y| \omega(|x-y|)} \leq M.
$$
This shows $(1)$. The proof of $(2)$ is also easy and is left to the reader.
\end{proof}

The first of our main results is the following.
\begin{theorem}\label{C1omega convex extension}
Let $\omega$ be a modulus of continuity. Let $C$ be a (not necessarily convex) subset of $\R^n$ . Let $f:C\to\R$ be an arbitrary function, and $G:C\to\R^n$ be continuous, with modulus of continuity $\omega$. Then $f$ has a convex, $C^{1,\omega}$ extension $F$ to all of $\R^n$, with $\nabla F=G$ on $C$, if and only if $(f,G)$ satisfies $(CW^{1,\omega})$ on $C$.
\end{theorem}
In particular, for the most important case that $\omega(t)=t$, we have the following.
\begin{corollary}\label{C11 convex extension}
Let $C$ be a (not necessarily convex) subset of $\R^n$. Let $f:C\to\R$ be an arbitrary function, and $G:C\to\R^n$ be a Lipschitz function. Then $f$ has a convex, $C^{1,1}$ extension $F$ to all of $\R^n$, with $\nabla F=G$ on $C$, if and only if $(f,G)$ satisfies $(CW^{1,1})$ on $C$.
\end{corollary}

It is worth noting that our proofs provide good control of the modulus of continuity of the gradients of the extensions $F$, in terms of that of $G$. In fact, assuming $\eta=1/2$ in $(CW^{1, \omega})$, which we can fairly do (see Proposition \ref{necessity} below), there exists a constant $k(n)>0$, depending only on the dimension $n$, such that
\begin{equation}\label{estimation of norm C11}
M(\nabla F,\Rn):= \sup_{x,y\in\Rn, \, x\neq y}\frac{|\nabla F(x)-\nabla F(y)|}{\omega(|x-y|)}  \leq k(n) \, M(G, C).
\end{equation}
Because convex functions on $\Rn$ are not bounded (unless they are constant), the most usual definitions of norms in the space $C^{1,\omega}(\Rn)$ are not suited to estimate convex functions. In this paper, for a differentiable function $F:\Rn\to\R$ we will denote
\begin{equation} \label{definitionnorm}
\|F\|_{1, \omega}=|F(0)|+|\nabla F(0)|+\sup_{x,y\in\Rn, \, x\neq y}\frac{|\nabla F(x)-\nabla F(y)|}{\omega(|x-y|)}.
\end{equation}
With this notation, and assuming $0\in C$ and $\eta=1/2$ in $(CW^{1,\omega})$, equation \eqref{estimation of norm C11} implies that
\begin{equation}\label{good control of norms}
\|F\|_{1,\omega}\leq k(n) \left(|f(0)|+|G(0)|+M(G,C) \right).
\end{equation}
In particular, the norm of the extension $F$ of $f$ that we construct is nearly optimal, in the sense that 
\begin{equation}\label{good control of norms 2}
\|F\|_{1,\omega}\leq k(n)\,
\inf\{\|\varphi\|_{1,\omega} : \varphi\in C^{1, \omega}(\R^n), \varphi_{|_C}=f, (\nabla\varphi)_{|_C}=G\}
\end{equation}
for a constant $k(n)\geq 1$ only depending on $n$.

In practice, it can happen that one is able to show that a pair $(f,G)$ satisfies $(CW^{1, \omega})$ with $M=M(G,C)$ and with $\eta<1/2$, but not with $\eta=1/2$. That is, $\eta=1/2$, though always theoretically achievable according to Proposition \ref{necessity}, may not be easily achievable in certain examples. The proof of Theorem \ref{C1omega convex extension} that we give for $\eta=1/2$ can be appropriately modified with some obvious changes to show the following quantitative version of this result, which may turn out to be more useful in such situations. 

\begin{theorem}\label{first quantitative version of C11 convex extension} 
Let $\omega$ be a modulus of continuity. Let $C$ be a (not necessarily convex) subset of $\R^n$ . Let $f:C\to\R$ be an arbitrary function, and $G:C\to\R^n$ be continuous, with modulus of continuity $\omega$.  Assume that 
$$
f(x)-f(y)-\langle G(y), x-y\rangle \geq \eta \, |G(x)-G(y)| \, \omega^{-1}\left( \frac{1}{2M} |G(x)-G(y)| \right)
$$
for all $x, y\in C$, where $M=M(G, C)$ and $\eta\in (0, 1/2]$.  
Then $f$ has a convex, $C^{1,\omega}$ extension $F$ to all of $\R^n$, with $\nabla F=G$ on $C$, and such that 
\begin{equation}\label{estimation of norm C11 for eta less that one half}
M(\nabla F,\Rn):= \sup_{x,y\in\Rn, \, x\neq y}\frac{|\nabla F(x)-\nabla F(y)|}{\omega(|x-y|)}  \leq k(n, \eta) \, M(G, C),
\end{equation}
where $k(n, \eta)$ is a constant depending only on the dimension $n$ and the number $\eta$.
\end{theorem}

Another quantitative approach to Theorem \ref{C1omega convex extension} consists in allowing $M>M(G,C)$ and getting rid of $\eta$ altogether in condition $(CW^{1, \omega})$ as follows. Assuming that $G$ is not constant, let $M_{*}(f,G,C)$ denote the smallest constant $M$ such that
$$
|G(x)-G(y)|\leq M\omega (|x-y|)
$$
and
$$
f(x)-f(y)-\langle G(y), x-y\rangle \geq |G(x)-G(y)|\omega^{-1}\left( \frac{1}{2M}|G(x)-G(y)|\right)
$$
for all $x,y\in C$. Again the same proof as that of Theorem \ref{C1omega convex extension}, with obvious changes, yields the following.
\begin{theorem}\label{second quantitative version of C11 convex extension} 
Let $\omega$ be a modulus of continuity, $C$ be a (not necessarily convex) subset of $\R^n$, and $f:C\to\R$ and $G:C\to\R^n$ be given mappings. 
Then $f$ has a convex, $C^{1,\omega}$ extension $F$ to all of $\R^n$, with $\nabla F=G$ on $C$, if and only if $M_{*}(f,G,C)<\infty$. In this case, $F$ can be taken so that 
\begin{equation}\label{estimation of norm C11 for Mstar}
M(\nabla F,\Rn):= \sup_{x,y\in\Rn, \, x\neq y}\frac{|\nabla F(x)-\nabla F(y)|}{\omega(|x-y|)}  \leq k(n) M_{*}(f, G, C),
\end{equation}
and, further assuming that $0\in C$, also
\begin{equation}\label{good control of norms for Mstar}
\|F\|_{1,\omega}\leq k(n) \left(|f(0)|+|G(0)|+M_{*}(f,G,C) \right),
\end{equation}
where $k(n)$ is a constant depending only on the dimension $n$.
\end{theorem}

In view of the following Remark, Theorem \ref{second quantitative version of C11 convex extension} is, at least formally, an improvement of Theorem \ref{first quantitative version of C11 convex extension}.
\begin{remark}
If a pair $(f,G)$ satisfies inequality $(CW^{1,\omega})$ on $C$ for some $\eta\in (0,1/2]$ and $M\geq M(G,C)$, then $M_{*}(f,G,C)\leq 2M/\eta$, and in particular $M_{*}(f,G,C)<\infty$.
\end{remark}
\begin{proof}
Using the fact that $\omega^{-1}(\mu s)\leq \mu\omega^{-1}(s)$ for $\mu\in [0,1]$, we have
\begin{eqnarray*}
& f(x)-f(y)-\langle G(y), x-y\rangle & \geq \eta |G(x)-G(y)| \omega^{-1}\left(\frac{1}{2M}|G(x)-G(y)|\right)\\
& & \geq  |G(x)-G(y)| \omega^{-1}\left(\frac{\eta}{2M}|G(x)-G(y)|\right),
\end{eqnarray*}
and by combining with Remark \ref{CW11 implies W11}(2) we obtain $M_{*}(f,G,C)\leq 2M/\eta$.
\end{proof}

Let us now consider a similar extension problem for the class of $C^1$ convex functions: given a continuous mapping $G:C\to\R^n$ and a function $f:C\to\R$, how can we decide whether there is a convex function $F\in C^1(\R^n)$ such that $F_{|_C}=f$ and $(\nabla F)_{|_C}=G$? There is evidence suggesting that, if $C$ is not assumed to be compact or $G$ is not uniformly continuous, this problem does not have a solution which is simple enough to use; see \cite[Example 4]{SchulzSchwartz}, \cite[Example 3.2]{VeselyZajicek}, and \cite[Example 4.1]{AM}. These examples show in particular that there exists a closed convex set $V\subset\R^2$ with nonempty interior and a $C^{\infty}$ function $f:\R^2\to\R$ so that $f$ is convex on an open convex neighborhood of $V$ and yet there is no convex function $F:\R^2\to\R$ such that $F=f$ on $V$.
Such $V$ and $f$ may be defined for instance by
$$V=\{(x,y)\in\R^2 \, : \,  x>0, xy\geq 1\},$$ and $$f(x,y)=-2\sqrt{xy} +\frac{1}{x+1} +\frac{1}{y+1}$$
for every $(x,y)\in V$. It is not difficult to see that the function $f$ admits a $C^{\infty}$ extension to $\R^2$, but there is no convex extension of $f$ defined on $\R^2$.
Nevertheless, we will show that there cannot be any such examples with $V$ compact (see Theorem \ref{main theorem C1} below). 

Since for a function $\varphi\in C^{1}(\R^n)$ and a compact set $C\subset\R^n$ there always exists a modulus of continuity for the restriction $(\nabla\varphi)_{|_C}$, Theorem \ref{C1omega convex extension} also provides a solution to our $C^1$ convex extension problem when $C$ is compact. However, given such a $1$-jet $(f, G)$ on a compact set $C$, unless $\omega(t)=t$ or one has a clue about what $\omega$ might do the job, in practice it may be difficult to find a modulus of continuity $\omega$ such that $(f,G)$ satisfies $(CW^{1, \omega})$, and for this reason it is also desirable to have a criterion for $C^1$ convex extendibility which does not involve dealing with moduli of continuity. We next study this question.

Given a $1$-jet $(f,G)$ on $C$ (where $f:C\to\R$ is a function and $G:C\to\R^n$ is a continuous mapping), a necessary condition for the existence of a convex function $F\in C^1(\R^n)$ with $F_{|_C}=f$ and $(\nabla F)_{|_C}=G$ is given by
$$
\lim_{|z-y|\to 0^{+}}\frac{f(z)-f(y)-\langle G(y), z-y\rangle}{|z-y|}=0 \textrm{ uniformly on } C, \eqno (W^1)
$$
which is equivalent to Whitney's classical condition for $C^1$ extendibility. If a $1$-jet $(f,G)$ satisfies condition $(W^1)$, Whitney's extension theorem \cite{Whitney} provides us with a function $F\in C^{1}(\R^n)$ such that $F_{|_C}=f$ and $(\nabla F)_{|_C}=G$. In the special case that $C$ is a convex body, if $f:C\to\R$ is convex and $(f,G)$ satisfies $(W^1)$, we will see that, without any further assumptions on $(f, G)$, $f$ {\em always} has a convex $C^1$ extension to all of $\R^n$ with $(\nabla F)_{|_C}=G$.

\begin{theorem}\label{main theorem C1}
Let $C$ be a compact convex subset of $\R^n$ with non-empty interior. Let $f:C\to\R$ be a convex function, and $G:C\to\R^n$ be a continuous mapping satisfying Whitney's extension condition $(W^1)$ on $C$. Then there exists a convex function $F\in C^{1}(\R^n)$ such that $F_{|_C}=f$ and $(\nabla F)_{|_C}=G$.
\end{theorem}

If $C$ and $f$ are convex but $\textrm{int}(C)$ is empty then, in order to obtain differentiable convex extensions of $f$ to all of $\R^n$ we will show that it is enough to complement $(W^1)$ with the following global geometrical condition:
$$
f(x)-f(y)=\langle G(y), x-y\rangle \implies G(x)= G(y), \textrm{ for all } x,y\in C. \eqno(CW^1)
$$

\begin{theorem}\label{main theorem C1 for C with empty interior}
Let $C$ be a compact convex subset of $\R^n$. Let $f:C\to\R$ be a convex function, and $G:C\to\R^n$ be a continuous mapping. Then $f$ has a convex, $C^1$ extension $F$ to all of $\R^n$, with $\nabla F=G$ on $C$, if and only if $f$ and $G$ satisfy $(W^1)$ and $(CW^1)$ on $C$.
\end{theorem}

In the general case of a non-convex compact set $C$, we will just have to add another global geometrical condition to $(CW^{1})$:
$$
f(x)-f(y)\geq \langle G(y), x-y\rangle \, \textrm{ for all } x,y\in C. \eqno(C)
$$
\begin{remark}\label{C implies W1}
If $(f,G)$ satisfies condition $(C)$ and $G$ is continuous, then $(f,G)$ satisfies Whitney's condition $(W^1)$.
\end{remark}
\noindent This is easily shown by an obvious modification of the proof of Remark \ref{CW11 implies W11}.

\begin{theorem}\label{main theorem C1 for nonconvex compacta}
Let $C$ be a compact (not necessarily convex) subset of $\R^n$. Let $f:C\to\R$ be an arbitrary function, and $G:C\to\R^n$ be a continuous mapping. Then $f$ has a convex, $C^1$ extension $F$ to all of $\R^n$, with $\nabla F=G$ on $C$, if and only if $(f,G)$ satisfies conditions $(C)$ and $(CW^1)$ on $C$.
\end{theorem}

Similarly to the $C^{1, \omega}$ case, we will see that the proof of the above result provides good control of the Lipschitz constant of the extension $F$ in terms of $\|G\|_{\infty}$. Namely, we will see that
\begin{equation}\label{estimation of Lip F in terms of bound for G}
\sup_{x\in\R^n}|\nabla F(x)|\leq k(n)\, \sup_{y\in C}|G(y)|
\end{equation}
for a constant $k(n) \geq 1$ only depending on $n$. Interestingly, this kind of control of $\textrm{Lip}(F)$ in terms of $\|G\|_{\infty}$ cannot be obtained, in general, for jets $(f,G)$ not satisfying $(C)$, as is easily seen by examples, and the proof of Whitney's extension theorem only permits to obtain extensions $(F, \nabla F)$ (of jets $(f,G)$ on $C$) which satisfy estimations of the type
$$
\sup_{x\in\R^n}|\nabla F(x)|\leq k(n)\, \left( \sup_{z, y\in C, z\neq y}\frac{|f(z)-f(y)|}{|z-y|}  + \sup_{y\in C}|G(y)| \right).
$$
or of the type
$$
\sup_{x\in\R^n}|\nabla F(x)|\leq k(n)\, \left( \sup_{y\in C}|f(y)|  + \sup_{y\in C}|G(y)| \right).
$$
It is condition $(C)$ that allows us to get finer control of $\textrm{Lip}(F)$ in the convex case; see the proof of Claim \ref{control of lip constant of Whitney extension for convex} below. In particular, assuming $0\in C$ and defining
\begin{equation}\label{definition of norm 1}
\|F\|_{1}:=|F(0)|+\sup_{x\in\R^n}|\nabla F(x)|,
\end{equation}
we obtain an extension $F$ of $f$ such that
\begin{equation}\label{good control of norms 3}
\|F\|_{1}\leq k(n) \,
\inf\{\|\varphi\|_{1} : \varphi\in C^{1}(\R^n), \varphi_{|_C}=f, (\nabla \varphi)_{|_C}=G\},
\end{equation}
so the norm of our extension is nearly optimal in this case too.

In the particular case when $C$ is finite, Theorem \ref{main theorem C1 for nonconvex compacta} provides necessary and sufficient conditions for interpolation of finite sets of data by $C^1$ convex functions.

\begin{corollary}\label{interpolation of data by C1 convex functions}
Let $S$ be a finite subset of $\R^n$, and $f:S\to\R$ be a function. Then there exists a convex function $F\in C^{1}(\R^n)$ with $F=f$ on $S$ if and only if there exists a mapping
$G:S\to\R^n$ such that $f$ and $G$ satisfy conditions $(C)$ and $(CW^1)$ on $S$.
\end{corollary}
In \cite[Theorem 14]{Mulansky} it is proved that, for every finite set of {\em strictly convex data} in $\R^n$ there always exists a $C^\infty$ convex function (or even a convex polynomial) that interpolates the given data. However, in the case that the data are convex but not strictly convex, the above corollary seems to be new.

Let us conclude this introduction with two geometrical applications of Corollary \ref{C11 convex extension} and Theorem \ref{main theorem C1 for nonconvex compacta} concerning characterizations of compact subsets $K$ of $\Rn$ which can be interpolated by
boundaries of $C^{1,1}$ or $C^1$ convex bodies (with prescribed unit outer normals on $K$). Namely, if
$K$ is a compact subset of $\R^n$ and we are given an $M$-Lipschitz (resp. continuous) map $N:K\to\R^n$ such that $|N(y)|=1$ for every $y\in K$, it is natural to ask what conditions on $K$ and $N$ are necessary and sufficient for $K$ to be a subset of the boundary of a $C^{1,1}$ (resp. $C^1$) convex body $V$ such that $0\in\textrm{int}(V)$ and $N(y)$ is outwardly normal to $\partial V$ at $y$ for every $y\in K$. A suitable set of conditions in the $C^{1,1}$ case is:
\begin{align*}
& (\mathcal{O}) & \langle N(y), y \rangle >0 \textrm{ for all } y\in K; \\
& (\mathcal{K}\mathcal{W}^{1,1}) & \langle N(y), y-x\rangle \geq \frac{\eta}{2M} |N(y)-N(x)|^2 \textrm{ for all } x, y\in K,
\end{align*}
for some $\eta\in (0,\frac{1}{2}]$. Our main result in this direction is as follows.
\begin{theorem}\label{corollary for C11 convex bodies}
Let $K$ be a compact subset of $\R^n$, and let $N:K\to\R^n$ be an $M$-Lipschitz mapping such that $|N(y)|=1$ for every $y\in K$.
Then the following statements are equivalent:
\begin{enumerate}
\item There exists a $C^{1,1}$ convex body $V$ with $0\in\textrm{int}(V)$ and such that $K\subseteq \partial V$ and $N(y)$ is outwardly normal to $\partial V$ at $y$ for every $y\in K$.
\item  $K$ and $N$ satisfy conditions $(\mathcal{O})$ and  $(\mathcal{K}\mathcal{W}^{1,1})$.
\end{enumerate}
\end{theorem}
This result may be compared to \cite{GhomiJDG2001}, where M. Ghomi showed how to construct $C^{m}$ smooth strongly convex bodies $V$ with prescribed strongly convex submanifolds and tangent planes. In the same spirit, the above Theorem allows us to deal with arbitrary compacta instead of manifolds, and to drop the strong convexity assumption, in the particular case of $C^{1,1}$ bodies. Similarly, for interpolation by $C^1$ bodies, the pertinent conditions are:
\begin{align*}
& (\mathcal{O}) & \langle N(y), y \rangle >0 \textrm{ for all } y\in K; \\
& (\mathcal{K}) & \langle N(y), x-y\rangle\leq 0 \textrm{ for all } x, y\in K; \\
& (\mathcal{K}\mathcal{W}^1) & \langle N(y), x-y\rangle=0 \implies N(x)=N(y) \textrm{ for all } x, y\in K,
\end{align*}
and our result for the class $C^1$ then reads as follows.
\begin{theorem}\label{corollary for C1 convex bodies}
Let $K$ be a compact subset of $\R^n$, and let $N:K\to\R^n$ be a continuous mapping such that $|N(y)|=1$ for every $y\in K$.
Then the following statements are equivalent:
\begin{enumerate}
\item There exists a $C^1$ convex body $V$ with $0\in\textrm{int}(V)$ and such that $K\subseteq \partial V$ and $N(y)$ is outwardly normal to $\partial V$ at $y$ for every $y\in K$.
\item $K$ and $N$ satisfy conditions $(\mathcal{O})$, $(\mathcal{K})$, and $(\mathcal{K}\mathcal{W}^1)$.
\end{enumerate}
\end{theorem}

The rest of this paper is devoted to the proofs of the above results. Most of the main ideas in the proofs of Theorem \ref{C1omega convex extension} and \ref{main theorem C1 for nonconvex compacta} are similar, but the case $C^{1, \omega}$ is considerably more technical, so, in order to convey these ideas more easily, we will begin by proving Theorems \ref{main theorem C1}, \ref{main theorem C1 for C with empty interior}, \ref{main theorem C1 for nonconvex compacta} and \ref{corollary for C1 convex bodies} in Section 2.
The proofs of Theorems \ref{C1omega convex extension} and \ref{corollary for C11 convex bodies} will be provided in Section 3. 

\medskip

\section{Proofs of the $C^{1}$ results}

Theorem \ref{main theorem C1} is a consequence of Theorem \ref{main theorem C1 for C with empty interior} and of the following result.

\begin{lemma}
Let $f\in C^{1}(\R^n)$, $C\subset\R^n$ be a compact convex set with nonempty interior, $x_0, y_0\in C$. Assume that $f$ is convex on $C$ and
$$
f(x_0)-f(y_0)=\langle \nabla f(y_0), x_0-y_0\rangle.
$$
Then $\nabla f(x_0)=\nabla f(y_0)$.
\end{lemma}
\begin{proof}
{\bf Case 1.} Suppose first that $f(x_0)=f(y_0)=0$. We may of course assume that $x_0\neq y_0$ as well. Then we also have $\langle \nabla f(y_0), x_0-y_0\rangle=0$. If we consider the $C^1$ function $\varphi(t)=f\left(y_0+t(x_0-y_0)\right)$, we have that $\varphi$ is convex on the interval $[0,1]$ and $\varphi'(0)=0$, hence $0=\varphi(0)=\min_{t\in[0,1]}\varphi(t)$, and because
$\varphi(0)=\varphi(1)$ and the set of minima of a convex function on a convex set is convex, we deduce that $\varphi(t)=0$ for all $t\in [0,1]$. This shows that $f$ is constant on the segment $[x_0, y_0]$ and in particular we have
$$
\langle \nabla f(z), z_0-z'_{0}\rangle=0 \textrm{ for all } z, z_0, z'_0\in [x_0, y_0].
$$
Now pick a point $a_0$ in the interior of $C$ and a number $r_0>0$ so that $B(a_0, r_0)\subset \textrm{int}(C)$. Since $C$ is a compact convex body, every ray emanating from a point $a\in B(a_0, r_0)$ intersects the boundary of $C$ at exactly one point. This implies that (even though the segment $[x_0, y_0]$ might entirely lie on the boundary $\partial C$),
for every $a\in B(a_0, r_0)$, the interior of the triangle $\Delta_a$ with vertices $x_0, a, y_0$, relative to the affine plane spanned by these points, is contained in the interior of $C$; we will denote $\textrm{relint}\left(\Delta_a\right)\subset\textrm{int}(C)$.

Let $p_0$ be the unique point in $[x_0, y_0]$ such that $|a_0-p_0|=d\left(a_0, [x_0, y_0]\right)$ (the distance to the segment $[x_0, y_0]$), set $w_0=a_0-p_0$, and denote $v_{a}:=a-p_0$ for each $a\in B(a_0, r_0)$. Thus for every $a\in B(a_0, r_0)$ we can write $v_a=u_a+w_0$, where $u_a:=a-a_0\in B(0, r_0)$, and in particular we have $\{v_a : a\in B(a_0, r_0)\}=B(w_0, r_0)$.

\begin{claim}
For every $z_0, z'_0$ in the relative interior of the segment $[x_0, y_0]$, we have $\nabla f(z_0)=\nabla f(z'_0)$.
\end{claim}
Let us prove our claim. It is enough to show that $\langle \nabla f(z_0)-\nabla f(z'_0), v_a\rangle=0$ for every $a\in B(a_0, r_0)$ (because if a linear form vanishes on a set with nonempty interior, such as $B(w_0, r_0)$, then it vanishes everywhere). So take $a\in B(a_0, r_0)$. Since $z_0$ and $z'_0$ are in the relative interior of the segment $[x_0, y_0]$ and $\textrm{relint}\left(\Delta_a\right)\subset \textrm{int}(C)$, there exists $t_0>0$ such that $z_0+tv_a, z'_0+tv_a \in\textrm{int}(C)$ for every $t\in (0, t_0]$.

If we had $\langle \nabla f(z'_0)-\nabla f(z_0), v_a\rangle>0$ then, because $f$ is convex on $C$ and $f(z_0)=f(z'_0)=0$, $\langle \nabla f(z'_0), z_0-z'_0\rangle =0$, we would get
$$
f(z_0+tv_a)=f(z'_0+z_0-z'_0+tv_a)\geq \langle \nabla f(z'_0), z_0-z'_0+tv_a\rangle=\langle \nabla f(z'_0), tv_a\rangle ,
$$
hence
$$
\lim_{t\to 0^{+}}\frac{f(z_0+tv_a)}{t}\geq \langle \nabla f(z'_0), v_a\rangle > \langle \nabla f(z_0), v_a\rangle = \lim_{t\to 0^{+}}\frac{f(z_0+tv_a)}{t},
$$
a contradiction. By interchanging the roles of $z_0, z'_0$, we see that the inequality $\langle \nabla f(z'_0)-\nabla f(z_0), v_a\rangle<0$ also leads to a contradiction. Therefore
$\langle \nabla f(z'_0)-\nabla f(z_0), v_a\rangle=0$ and the Claim is proved.

Now, by using the continuity of $\nabla f$, we easily conclude the proof of the Lemma in Case 1.

\medskip

\noindent {\bf Case 2.} In the general situation, let us consider the function $h$ defined by
$$
h(x)=f(x)-f(y_0)-\langle \nabla f(y_0), x-y_0\rangle ,  \,\,\, x\in\R^n.
$$
It is clear that $h$ is convex on $C$, and $h\in C^{1}(\R^n)$. We also have
$$\nabla h(x)=\nabla f(x)-\nabla f(y_0),
$$
and in particular $\nabla h(y_0)=0$. Besides, using the assumption
that $f(x_0)-f(y_0)=\langle \nabla f(y_0), x_0-y_0\rangle$, we have $h(x_0)=0=h(y_0)$, and $h(x_0)-h(y_0)=\langle \nabla h(y_0), x_0-y_0\rangle$. Therefore we can apply Case 1 with $h$ instead of $f$ and we get that $\nabla h(x_0)=\nabla h(y_0)=0$, which implies that $\nabla f(x_0)=\nabla f(y_0)$.
\end{proof}

\medskip

From the above Lemma it is clear that $(CW^1)$ is a necessary condition for a convex function $f:C\to\R$ (and a mapping $G:C\to\R^n$) to have a convex, $C^1$ extension $F$ to all of $\R^n$ with $\nabla F=G$ on $C$, and also that if the jet $(f,G)$ satisfies $(W^1)$ and $\textrm{int}(C)\neq\emptyset$ then $(f,G)$ automatically satisfies $(CW^1)$ on $C$ as well. It is also obvious that Theorem \ref{main theorem C1 for C with empty interior} is an immediate consequence of Theorem \ref{main theorem C1 for nonconvex compacta}, and that the condition $(C)$ is also necessary in Theorem \ref{main theorem C1 for nonconvex compacta}. Thus, in order to prove Theorems \ref{main theorem C1}, \ref{main theorem C1 for C with empty interior}, and \ref{main theorem C1 for nonconvex compacta} it will be sufficient to establish the {\em if} part of Theorem \ref{main theorem C1 for nonconvex compacta}.

\subsection{Proof of Theorem \ref{main theorem C1 for nonconvex compacta}.}

Because $f$ satisfies $(C)$ and $G$ is continuous, by Remark \ref{C implies W1} we know that $(f,G)$ satisfies $(W^1)$. Then, according to Whitney's Extension Theorem, there exists $\widetilde{f}\in C^{1}(\R^n)$ such that, on $C$, we have $\widetilde{f}=f$ and $\nabla \widetilde{f}=G$. 
\begin{claim}\label{control of lip constant of Whitney extension for convex}
If $f$ satisfies $(C)$ then we can further assume that there exists a constant $k(n)$, only depending on $n$, such that 
\begin{equation}
\textrm{Lip}(\widetilde{f})=\sup_{x\in\R^n}|\nabla \widetilde{f}(x)|\leq k(n)\, \sup_{y\in C}|G(y)|.
\end{equation}
\end{claim}
\begin{proof}
Let us recall the construction of the function $\tilde{f}$. Consider a Whitney's partition of unity $\lbrace \varphi_j \rbrace_j$ associated to the family of Whitney's cubes $\lbrace Q_j, \: Q_j^* \rbrace_j$ decomposing $\Rn \setminus C$ (see the following section or \cite[Chapter VI]{Stein} for notation). Define polynomials $P_{z}(x) = f(z) + \langle G(z), x-z \rangle$ for all $x\in \Rn$ and all $z\in C.$ For every $j,$ find a point $p_j\in C$ such that $d(C,Q_j)=d(p_j,Q_j).$ Then define the function $\tilde{f}$ by
\begin{equation}\label{definition of the Whitney extension}
\tilde{f}(x)=  \left\lbrace
	\begin{array}{ccl}
	\sum_{j} P_{p_j}(x) \varphi_j(x) & \mbox{ if } x\in \R^n\setminus C \\
	f(x) & \mbox{ if }  x\in C.
	\end{array}
	\right.
\end{equation}
As a particular case of the proof of the Whitney extension theorem, we know that this function $\tilde{f}$ is of class $C^1(\Rn)$, extends $f$ to $\Rn$ and satisfies $\nabla \tilde{f} = G$ on $C.$ By the definition of $\tilde{f}$ we can write, for $x \in \Rn \setminus C,$
\begin{equation}\label{derivative of Whitneys extension}
  \nabla \tilde{f}(x)  = \sum_{j} \nabla P_{p_j}(x) \varphi_j(x) + \sum_{j} P_{p_j}(x) \nabla \varphi_j(x).
\end{equation}
Since $\nabla P_{p_j} = G(p_j)$ for all $j,$ the first sum is bounded above by $\sup \lbrace |G(y)| \: : \: y\in C \rbrace :=\|G\|_{\infty}$. In order to estimate the second sum, recall that $\sum_j \nabla \varphi_j = 0$ and find a point $b\in C$ such that $|b-x|= d(x,C)$. Then we write
\begin{equation}\label{trick 1 in estimation of Lip(f)}
 \sum_{j} P_{p_j}(x) \nabla \varphi_j(x) =  \sum_{j} \left(  P_{p_j}(x) -P_b(x) \right) \nabla \varphi_j(x),
\end{equation}
and observe that 
\begin{align*}
P_{p_j}(x)-P_b(x) & = f(p_j)+\langle G(p_j) , x-p_j \rangle - f(b)-\langle G(b), x-b \rangle \\
& = f(p_j)-f(b)-\langle G(b), p_j-b \rangle + \langle G(p_j)-G(b), x-p_j \rangle.
\end{align*}
By the same argument used in Remark \ref{C implies W1} involving condition $(C),$ we have that
\begin{eqnarray*}
& & 0 \leq f(p_j)-f(b)-\langle G(b), p_j-b \rangle \\
& & \leq \langle G(b)-G(p_j), b-p_j \rangle \leq 2 \|G\|_\infty |b-p_j|.
\end{eqnarray*}
On the other hand,
$$
|  \langle G(p_j)-G(b), x-p_j \rangle | \leq 2 \| G \|_\infty |x-p_j|.
$$
These inequalities lead us to 
\begin{equation}
| P_{p_j}(x)-P_b(x) | \leq 2 \| G \|_\infty ( |b-p_j| + |x-p_j| ).
\end{equation} 
For those integers $j$ such that $x\in Q_{j}^{*}$, the results exposed in \cite[Chapter VI]{Stein} show that $|b-p_j| \leq 8 |x-p_j|$, and that $|x-p_j| $ is of the same order as $\diam(Q_j)$ (with constants not even depending on $n$). Hence 
$$
| P_{p_j}(x)-P_b(x) | \lesssim  \| G \|_\infty \diam(Q_j) \quad \text{if} \quad x\in Q_j^*
$$
(by $A\lesssim B$ we mean that $A\leq K B$, where $K$ is a constant only depending on the dimension $n$).
Also, by the properties of the Whitney's partition of unity $\lbrace \varphi_j \rbrace_j$ we know that
$$
 \big| \nabla \varphi_j(x) \big | \lesssim \diam(Q_j)^{-1},
 $$
and because all these sums has at most $N=(12)^n$ nonzero terms, we obtain
 $$
\sum_{Q_j^* \ni x} | P_{p_j}(x)-P_b(x) | \big| \nabla \varphi_j(x) \big | \lesssim  \| G \|_\infty  \sum_{Q_j^* \ni x} \diam(Q_j) \diam(Q_j)^{-1} \lesssim \| G \|_\infty,
$$
which together with \eqref{trick 1 in estimation of Lip(f)} allows us to control the second sum in \eqref{derivative of Whitneys extension} as required.
\end{proof}
Thus we may and do assume in what follows, for simplicity of notation, that $f$ is of class $C^{1}(\R^n)$, with $\nabla f=G$ on $C$, and that $f$ satisfies conditions $(C)$ and $(CW^1)$ on $C$. Occasionally, if the distinction between $\widetilde{f}$ and $f$ matters (e.g. in the estimations of Lipschitz constants involving $\widetilde{f}$), we will nevertheless write $\widetilde{f}$ instead of $f$ in order to prevent any misinterpretation.
Let us consider the function $m(f):\R^n\to\R$ defined by
\begin{equation}
m(f)(x)=\sup_{y\in C}\{f(y)+\langle \nabla f(y), x-y\rangle\}.
\end{equation}
Since $C$ is compact and the function $y\mapsto f(y)+\langle \nabla f(y), x-y\rangle$ is continuous, it is obvious that $m(f)(x)$ is well defined, and in fact the sup is attained, for every $x\in \R^n$. Furthermore, if we set
\begin{equation}
K:=\max_{y\in C}|\nabla f(y)|=\max_{y\in C}|G(y)|
\end{equation}
then each affine function $x\mapsto f(y)+\langle \nabla f(y), x-y\rangle$ is $K$-Lipschitz, and therefore $m(f)$, being a sup of a family of convex and $K$-Lipschitz functions, is convex and $K$-Lipschitz on $\R^n$. Note also that
\begin{equation}
K\leq \textrm{Lip}(\widetilde{f}).
\end{equation}
Moreover, we have 
\begin{equation}
m(f)=f \textrm{ on } C.
\end{equation}
Indeed, if $x\in C$ then, because $f$ satisfies $(C)$ on $C$, we have $f(x)\geq f(y)+\langle \nabla f(y), x-y\rangle$ for every $y\in C$, hence $m(f)(x)\leq f(x)$. On the other hand, we also have $f(x)\leq m(f)(x)$ because of the definition of $m(f)(x)$ and the fact that $x\in C$.

(In the case when $C$ is convex and has nonempty interior, it is easy to see that if $h:\R^n\to\R$ is convex and $h=f$ on $C$, then $m(f)\leq h$. Thus, in this case, $m(f)$ is the minimal convex extension of $f$ to all of $\R^n$, which accounts for our choice of notation. However, if $C$ is convex but has empty interior then there is no minimal convex extension operator. We refer the interested reader to \cite{SchulzSchwartz} for necessary and sufficient conditions for $m(f)$ to be finite everywhere, in the situation when $f:C\to\R$ is convex but not necessarily everywhere differentiable.)

If the function $m(f)$ were differentiable on $\R^n$, there would be nothing else to say. Unfortunately, it is not difficult to construct examples showing that $m(f)$ need not be differentiable outside $C$ (even when $C$ is convex and $f$ satisfies $(CW^1)$, see Example \ref{nonsmoothness of m(f)} at the end of this section). Nevertheless, a crucial step in our proof is the following fact: $m(f)$ is differentiable on $C$, provided that $f$ satisfies conditions $(C)$ and $(CW^1)$ on $C$.

\begin{lemma}\label{CW1 implies differentiability of m(f) on C}
Let $f\in C^{1}(\R^n)$, let $C$ be a compact subset of $\R^n$ (not necessarily convex), and assume that $f$ satisfies $(C)$ and $(CW^1)$ on $C$. Then, for each $x_0\in C$, the function $m(f)$ is differentiable at $x_0$, with $\nabla m(f)(x_0)=\nabla f(x_0)$.
\end{lemma}
\begin{proof}
Notice that, by definition of $m(f)$ we have, for every $x\in\R^n$,
$$
\langle \nabla f(x_0), x-x_0\rangle + m(f)(x_0)= \langle \nabla f(x_0), x-x_0\rangle + f(x_0)\leq m(f)(x).
$$
Since $m(f)$ is convex, this means that $\nabla f(x_0)$ belongs to $\partial m(f)(x_0)$ (the subdifferential of $m(f)$ at $x_0$). If $m(f)$ were not differentiable at $x_0$ then there would exist a number $\varepsilon>0$ and a sequence $(h_k)$ converging to $0$ in $\R^n$ such that
\begin{equation}\label{contradict differentiability of m(f)}
\frac{ m(f)(x_0+h_k)-m(f)(x_0)-\langle \nabla f(x_0), h_k\rangle}{|h_k|}\geq\varepsilon \,\,\, \textrm{ for every } k\in\N.
\end{equation}
Because the sup defining $m(f)(x_0+h_k)$ is attained, we obtain a sequence $(y_k)\subset C$ such that
$$
m(f)(x_0+h_k)=f(y_k)+\langle \nabla f(y_k), x_0+h_k-y_k\rangle,
$$
and by compactness of $C$ we may assume, up to passing to a subsequence, that $(y_k)$ converges to some point $y_0\in C$. Because $f=m(f)$ on $C$, and by continuity of $f$, $\nabla f$, and $m(f)$ we then have
\begin{eqnarray*}
& &f(x_0)=m(f)(x_0)=\lim_{k\to\infty}m(f)(x_0+h_k)=\\
& &\lim_{k\to\infty} \left(f(y_k)+\langle \nabla f(y_k), x_0+h_k-y_k\rangle\right)= f(y_0)+\langle \nabla f(y_0), x_0-y_0\rangle,
\end{eqnarray*}
that is, $f(x_0)-f(y_0)=\langle \nabla f(y_0), x_0-y_0\rangle$. Since $x_0, y_0\in C$ and $f$ satisfies $(CW^1)$, this implies that $\nabla f(x_0)=\nabla f(y_0)$.
And because $m(f)(x_0)\geq f(y_k)+\langle \nabla f(y_k), x_0-y_k\rangle$ by definition of $m(f)$, we then have
\begin{eqnarray*}
& & \frac{ m(f)(x_0+h_k)-m(f)(x_0)-\langle \nabla f(x_0), h_k\rangle}{|h_k|}\leq\\
& & \frac{ f(y_k)+ \langle \nabla f(y_k), x_0+h_k-y_k\rangle -f(y_k)-\langle \nabla f(y_k), x_0-y_k\rangle -\langle \nabla f(x_0), h_k\rangle}{|h_k|}=\\
& &\frac{\langle \nabla f(y_k)-\nabla f(x_0), h_k\rangle}{|h_k|}\leq |\nabla f(y_k)-\nabla f(x_0)| = |\nabla f(y_k)-\nabla f(y_0)|,
\end{eqnarray*}
from which we deduce, using the continuity of $\nabla f$, that
$$
\limsup_{k\to\infty}\frac{ m(f)(x_0+h_k)-m(f)(x_0)-\langle \nabla f(x_0), h_k\rangle}{|h_k|}\leq 0,
$$
in contradiction with $(\ref{contradict differentiability of m(f)})$.
\end{proof}

Now we proceed with the rest of the proof of Theorem \ref{main theorem C1 for nonconvex compacta}. Our strategy will be to use the differentiability of $m(f)$ on $\partial C$ in order to construct a (not necessarily convex) differentiable function $g$ such that $g=f$ on $C$,
$g\geq m(f)$ on $\R^n$, and $\lim_{|x|\to\infty}g(x)=\infty$. Then we will define $F$ as the convex envelope of $g$, which will be of class $C^1(\R^n)$ and will coincide with $f$ on $C$.

For each $\varepsilon>0$, let $\theta_{\varepsilon}:\R\to\R$ be defined by
$$
\theta_{\varepsilon}(t)   =  \left\lbrace
	\begin{array}{ccl}
	0 & \mbox{ if } t\leq 0 \\
	t^2 & \mbox{ if } t\leq\frac{K+\varepsilon}{2} \\
	(K+\varepsilon)\left(t-\frac{K+\varepsilon}{2}\right) +\left(\frac{K+\varepsilon}{2}\right)^2  & \mbox{ if }  t>\frac{K+\varepsilon}{2}
	\end{array}
	\right.
$$
(recall that $K=\max_{y\in C}\|\nabla f (y)\|\leq \textrm{Lip}(\widetilde{f})$). Observe that $\theta_{\varepsilon}\in C^{1}(\R)$, $\textrm{Lip}(\theta_{\varepsilon})=K+\varepsilon$. Now set
$$
\Phi_{\varepsilon}(x)=\theta_{\varepsilon}\left(d(x, C)\right),
$$
where $d(x, C)$ stands for the distance from $x$ to $C$, notice that 
$\Phi_{\varepsilon}(x)=d(x,C)^2$ on an open neighborhood of $C$,
and
 define
$$
H_{\varepsilon}(x)=|f(x)-m(f)(x)|+ 2\Phi_{\varepsilon}(x).
$$ 
Note that $\textrm{Lip}(\Phi_{\varepsilon})=\textrm{Lip}(\theta_{\varepsilon})$ because $d(\cdot, C)$ is $1$-Lipschitz, and therefore
\begin{equation}
\textrm{Lip}(H_{\varepsilon})\leq \textrm{Lip}(\widetilde{f}) + K + 2(K+\varepsilon)\leq 4\, \textrm{Lip}(\widetilde{f}) +2\varepsilon.
\end{equation}
\begin{claim}
$H_{\varepsilon}$ is differentiable on $C$, with $\nabla H_{\varepsilon}(x_0)=0$ for every $x_0\in C$.
\end{claim}
\begin{proof}
The function $d(\cdot, C)^{2}$ is obviously differentiable, with a null gradient, at $x_0$, hence we only have to see that $|f-m(f)|$ is differentiable, with a null  gradient, at $x_0$. Since $\nabla m(f)(x_0)=\nabla f(x_0)$ by Lemma \ref{CW1 implies differentiability of m(f) on C}, the Claim boils down to the following easy exercise: if two functions $h_1, h_2$ are differentiable at $x_0$, with $\nabla h_1(x_0)=\nabla h_2 (x_0)$, then $|h_1-h_2|$ is differentiable, with a null gradient, at $x_0$.
\end{proof}
Now, because $\Phi_{\varepsilon}$ is continuous and positive on $\R^n\setminus C$, using mollifiers and a partition of unity, one can construct a function $\varphi_{\varepsilon}\in C^{\infty}(\R^n\setminus C)$ such that
\begin{equation}
|\varphi_{\varepsilon}(x)-H_{\varepsilon}(x)|\leq \Phi_{\varepsilon}(x) \,\,\, \textrm{ for every } x\in \R^{n}\setminus C,
\end{equation}
and 
\begin{equation}
\textrm{Lip}(\varphi_{\varepsilon})\leq \textrm{Lip}(H_{\varepsilon})+\varepsilon
\end{equation}
(see for instance \cite[Proposition 2.1]{GW} for a proof in the more general setting of Riemannian manifolds, or \cite{AFLR} even for possibly infinite-dimensional Riemannian manifolds).
Let us define $\widetilde{\varphi}=\widetilde{\varphi_{\varepsilon}}:\R^n\to\R$ by
$$
\widetilde{\varphi}=  \left\lbrace
	\begin{array}{ccl}
	\varphi_{\varepsilon}(x) & \mbox{ if } x\in \R^n\setminus C \\
	0  & \mbox{ if }  x\in C.
	\end{array}
	\right.
$$
\begin{claim}
The function $\widetilde{\varphi}$ is differentiable on $\R^n$, and it satisfies $\nabla\widetilde{\varphi}(x_0)=0$ for every $x_0\in C$.
\end{claim}
\begin{proof}
It is obvious that $\widetilde{\varphi}$ is differentiable on $\textrm{int}(C)\cup \left(\R^n\setminus C\right)$. We also have $\nabla\widetilde{\varphi}=0$ on $\textrm{int}(C)$, trivially. We only have to check that $\widetilde{\varphi}$ is differentiable, with a null gradient, on $\partial C$. If $x_0\in \partial C$ we have (recalling that $\Phi_{\varepsilon}(x)=d(x,C)^2$ on a neighborhood of $C$) that
$$
\frac{|\widetilde{\varphi}(x)-\widetilde{\varphi}(x_0)|}{|x-x_0|}=
\frac{|\widetilde{\varphi}(x)|}{|x-x_0|}\leq \frac{|H_{\varepsilon}(x)|+d(x, C)^2}{|x-x_0|}\to 0
$$
as $|x-x_0|\to 0^{+}$, because both $H_{\varepsilon}$ and $d(\cdot, C)^2$ vanish at $x_0$ and are differentiable, with null gradients, at $x_0$. Therefore $\widetilde{\varphi}$ is differentiable at $x_0$, with $\nabla \widetilde{\varphi}(x_0)=0$.
\end{proof}
Note also that
\begin{equation}\label{lip constant of varphi}
\textrm{Lip}(\widetilde{\varphi})=\textrm{Lip}(\varphi_{\varepsilon})\leq \textrm{Lip}(H_{\varepsilon})+\varepsilon\leq 4\, \textrm{Lip}(\widetilde{f})+3\varepsilon.
\end{equation}

Next we define 
\begin{equation}
g=g_{\varepsilon}:=f+\widetilde{\varphi}.
\end{equation}
The function $g$ is differentiable on $\R^n$, and coincides with $f$ on $C$. Moreover, we also have $\nabla g=\nabla f$ on $C$ (because $\nabla\widetilde{\varphi}=0$ on $C$).  And, for $x\in \R^n\setminus C$, we have
\begin{eqnarray*}
& &
g(x)\geq f(x)+H(x)-\Phi_{\varepsilon}(x)=f(x)+|f(x)-m(f)(x)|+\Phi_{\varepsilon}(x)\geq\\
& & m(f)(x) + \Phi_{\varepsilon}(x).
\end{eqnarray*}
This shows that $g\geq m(f)$. On the other hand, because $m(f)$ is $K$-Lipschitz, we have
$$
m(f)(x)\geq m(f)(0)-K|x|,
$$
and because $C$ is bounded, say $C\subset B(0,R)$ for some $R>0$, also
\begin{eqnarray*}
& & \Phi_{\varepsilon}(x)=(K+\varepsilon) d(x,C) - \frac{(K+\varepsilon)^2}{4} \\
& & \geq (K+\varepsilon) d(x, B(0,R)) - \frac{(K+\varepsilon)^2}{4} =(K+\varepsilon)\left(|x|-R-\frac{K+\varepsilon}{4}\right)
\end{eqnarray*}
for $|x|\geq R+ \frac{K+\varepsilon}{2}$.
Hence 
$$
g(x)\geq m(f)(x) + \Phi_{\varepsilon}(x)\geq m(f)(0)-K|x|+
(K+\varepsilon)\left(|x|-R-\frac{K+\varepsilon}{4}\right),
$$
for $|x|$ large enough, which implies 
\begin{equation}
\lim_{|x|\to\infty}g(x)=\infty.
\end{equation}
Also, notice that according to \eqref{lip constant of varphi} and the definition of $g$, we have
\begin{equation}
\textrm{Lip}(g)\leq \textrm{Lip}(\widetilde{f}) +\textrm{Lip}(\widetilde{\varphi})\leq 5\, \textrm{Lip}(\widetilde{f}) +3\varepsilon.
\end{equation}
Now we will use a differentiability property of the convex envelope of a function $\psi:\R^n\to\R$, defined by
$$
\textrm{conv}(\psi)(x)=\sup\{ h(x) \, : \, h \textrm{ is convex }, h\leq \psi\}
$$
(another expression for  $\textrm{conv}(\psi)$, which follows form Carath\'eodory's Theorem, is
$$
\textrm{conv}(\psi)(x)=\inf\left\lbrace \sum_{j=1}^{n+1}\lambda_{j} \psi(x_j) \, : \, \lambda_j\geq 0,
\sum_{j=1}^{n+1}\lambda_j =1, x=\sum_{j=1}^{n+1}\lambda_j x_j \right\rbrace,
$$
see \cite[Corollary 17.1.5]{Rockafellar} for instance). The following result is a restatement of a particular case of the main theorem in \cite{KirchheimKristensen}; see also \cite{GriewankRabier}.
\begin{theorem}[Kirchheim-Kristensen]
If $\psi:\R^n\to\R$ is differentiable and $\lim_{|x|\to\infty}\psi(x)=\infty$, then $\textrm{conv}(\psi)\in C^1(\R^n)$.
\end{theorem}
\noindent Although not explicitly stated in that paper, the proof of \cite{KirchheimKristensen} also shows that
$$
\textrm{Lip}\left(\textrm{conv}(\psi)\right)\leq\textrm{Lip}(\psi).
$$
If we define $F=\textrm{conv}(g)$ we thus get that $F$ is convex on $\R^n$ and $F\in C^{1}(\R^n)$, with 
\begin{equation}\label{lispchitz constant of F}
\textrm{Lip}(F)\leq\textrm{Lip}(g)\leq 5\textrm{Lip}(\widetilde{f})+3\varepsilon
\leq 5 \, k(n)\, \sup_{y\in C}|G(y)| +3\varepsilon.
\end{equation}
Let us now check that $F=f$ on $C$. Since $m(f)$ is convex on $\R^n$ and $m(f)\leq g$, we have that $m(f)\leq F$ on $\R^n$ by definition of $\textrm{conv}(g)$. On the other hand, since $g=f$ on $C$ we have, for every convex function $h$ with $h\leq g$, that $h\leq f$ on $C$, and therefore, for every $y\in C$,
$$
F(y)=\sup\{ h(y) \, : \, h \textrm{ is convex }, h\leq g\}\leq f(y)=m(f)(y).
$$
This shows that $F(y)=f(y)$ for every $y\in C$.

Next let us see that we also have $\nabla F(y)=\nabla f(y)$ for every $y\in C$. In order to do so we use the following well known criterion for differentiability of convex functions, whose proof is straightforward and can be left to the interested reader.
\begin{lemma}\label{differentiability criterion for convex functions}
If $\phi$ is convex, $\psi$ is differentiable at $y$, $\phi\leq \psi$, and $\phi(y)=\psi(y)$, then $\phi$ is differentiable at $y$, with $\nabla\phi(y)=\nabla\psi(y)$.
\end{lemma}
\noindent (This fact can also be phrased as: a convex function $\phi$ is differentiable at $y$ if and only if $\phi$ is superdifferentiable at $y$.)

Since we know that $m(f)\leq F$, $m(f)(y)=f(y)=F(y)$ for all $y\in C$, and $F\in C^{1}(\R^n)$, it follows from this criterion (by taking $\phi=m(f)$ and $\psi=F$), and from Lemma \ref{CW1 implies differentiability of m(f) on C}, that
$$
G(y)=\nabla f(y)=\nabla m(f)(y)=\nabla F(y) \, \textrm{ for all } y\in C.
$$
Finally, note that, equation \eqref{lispchitz constant of F} implies (by assuming that $\varepsilon\leq k(n) \|G\|_{\infty}/3$, which we may do) that
\begin{equation}
\textrm{Lip}(F)\leq 6\, k(n) \sup_{y\in C}|G(y)|
\end{equation}
and also, assuming $0\in C$, that
\begin{equation}
\|F\|_{1}\leq 6 \, k(n) \,
\inf\{\|\varphi\|_{1} : \varphi\in C^{1}(\R^n), \varphi_{|_C}=f, (\nabla \varphi)_{|_C}=G\}.
\end{equation}
The proof of Theorem \ref{main theorem C1 for nonconvex compacta} is complete.

\subsection{Proof of Theorem \ref{corollary for C1 convex bodies}.}
\noindent $(2)\implies (1)$: Set $C=K\cup\{0\}$, choose a number $\alpha>0$ sufficiently close to $1$ so that
$$
0<1-\alpha <\min_{y\in K}\langle N(y), y\rangle
$$
(this is possible thanks to condition $(\mathcal{O})$, continuity of $N$ and compactness of $K$; notice in particular that $0\notin K$),
and define $f:C\to \R$ and $G:C\to\R^n$ by
$$
 f(y)   =  \left\lbrace
	\begin{array}{ccl}
	1 & \mbox{ if } y\in K \\
	\alpha  & \mbox{ if }  y=0,
	\end{array}
	\right. \textrm{ and }
 G(y)   =  \left\lbrace
	\begin{array}{ccl}
	N(y) & \mbox{ if } y\in K \\
	0  & \mbox{ if }  y= 0.
	\end{array}
	\right.
$$
By using conditions $(\mathcal{K})$ and $(\mathcal{K}\mathcal{W}^1)$,
it is straightforward to check that $f$ and $G$ satisfy conditions $(C)$ and $(CW^1)$. Therefore, according to Theorem \ref{main theorem C1 for nonconvex compacta}, there exists a convex function $F\in C^1(\R^n)$ such that $F=f$ and $\nabla F=G$ on $C$. Moreover, from the proof of Theorem \ref{main theorem C1 for nonconvex compacta}, it is clear that $F$ can be taken so as to satisfy $\lim_{|x|\to\infty}F(x)=\infty$. If we define
$V=\{x\in\R^n \, : \, F(x)\leq 1\}$ we then have that $V$ is a compact convex body with $0\in\textrm{int}(V)$ (because $F(0)=\alpha<1$), and $\nabla F(y)=N(y)$ is outwardly normal to
$\{x\in\R^n \, : \, F(x)=1\}=\partial V$ at each $y\in K$. Moreover, $F(y)=f(y)=1$ for each $y\in K$, hence $K\subseteq \partial V$.

\medskip

\noindent $(1)\implies (2)$: Let $\mu_C$ be the Minkowski functional of $C$. By composing $\mu_C$ with a $C^1$ convex function $\phi:\R\to\R$ such that $\phi(t)=|t|$ if and only if $|t|\geq 1/2$, we obtain a function $F(x):=\phi(\mu_{C}(x))$ which is of class $C^1$ and convex on $\R^n$ and coincides with $\mu_C$ on a neighborhood of $\partial C$. By Theorem \ref{main theorem C1 for nonconvex compacta} we then have that $F$ satisfies conditions $(C)$ and $(CW^1)$ on $\partial C$. On the other hand $\nabla\mu_C(y)$ is outwardly normal to $\partial C$ at every $y\in\partial C$ and $\nabla F=\nabla \mu_C$ on $\partial C$, and hence we have $\nabla F(y)=N(y)$ for every $y\in \partial C$. These facts, together with the assumption $K\subseteq \partial V$, are easily checked to imply that $N$ and $K$ satisfy conditions $(\mathcal{K})$ and $(\mathcal{K}\mathcal{W}^1)$. Finally, because $0\in \textrm{int}(C)$ and $\nabla \mu_{C}(x)$ is outwardly normal to $\partial C$ for every $x\in\partial C$, we have that $\langle \nabla\mu_C(x), x\rangle >0$ for every $x\in\partial C$, and in particular condition $(\mathcal{O})$ is satisfied as well. $\Box$

\medskip

Let us conclude this section with a couple of examples. We first observe that $m(f)$ need not be differentiable outside $C$, even in the case when $C$ is a convex body and $f$ is $C^{\infty}$ on $C$.
\begin{example}\label{nonsmoothness of m(f)}
{\em Let $g$ be the function $g(x,y)=\max\{x+y-1, -x+y-1, \frac{1}{3}y\}$. Using for instance the smooth maxima introduced in \cite{Azagra}, one can smooth away the edges of the graph of $g$ produced by the intersection of the plane $z=\frac{1}{3}y$ with the planes $z=y \pm x-1$, thus obtaining a smooth convex function $f$ defined on $C:=g^{-1}(-\infty, 0]\cap\{(x,y): y\geq -1\}$.
However, $m(f)$ will not be everywhere differentiable, because for $y\geq 2$ we have $m(f)(x,y)=\max\{x+y-1, -x+y-1\}$, and this max function is not smooth on the line $x=0$. We leave the details to the interested reader.}
\end{example}

The following example shows that when $C$ has empty interior there are convex functions $f:C\to\R$ and continuous mappings $G:C\to\R^n$ which satisfy $(W^1)$ but do not satisfy $(CW^1)$.
\begin{example}
{\em Let $C$ be the segment $\{0\}\times [0,1]$ in $\R^2$, and $f$, $G$ be defined by $f(0,y)=0$ and $G(0,y)=(y,0)$. If we define $\widetilde{f}(x,y)=xy$ then it is clear that $\widetilde{f}$ is a $C^{1}$ extension of $f$ to $\R^2$ which satisfies $\nabla \widetilde{f}(0,y)=G(0,y)$ for $(0,y)\in C$. Therefore the pair $f,G$ satisfies Whitney's extension condition $(W^1)$. However, since $f$ is constant on the segment $C$ and $G(0,1)=(1,0)\neq (0,0)=G(0,0)$, it is clear that the pair $f, G$ does not satisfy $(CW^1)$. In particular $f$ does not have any convex $C^1$ extension $F$ to $\R^n$ with $\nabla F=G$ on $C$.}
\end{example}

\medskip

\section{Proofs of the $C^{1, \omega}$ results}

\subsection{Necessity}
Let us prove the necessity of condition $(CW^{1, \omega})$ in Theorem \ref{C1omega convex extension}.
We will use the following.
\begin{lemma}\label{h must be negative somewhere}
Let $\omega:[0,\infty)\to [0, \infty)$ be a modulus of continuity, let $a,b, \eta$ be real numbers with $a>0$, $\eta\in (0, \frac{1}{2}]$, and define $h:[0, \infty)\to\R$ by
$$
h(s)=-a s + b +\omega(s) s.
$$
Assume that $b<\eta a \, \omega^{-1}(\eta a)$. Then we have
$$
h\left( \omega^{-1}\left(\eta a\right)\right)<0.
$$
\end{lemma}
\begin{proof}
We can write
$$
h\left( \omega^{-1}\left(\eta a\right)\right)=-a(1-2\eta)\omega^{-1}(\eta a) + b - \eta a\,\omega^{-1}(\eta a),
$$
and the result follows at once.
\end{proof}

\begin{proposition}\label{necessity}
Let $f\in C^{1, \omega}(\Rn)$ be convex and not affine. Then
$$
f(x)-f(y)-\langle \nabla f(y), x-y\rangle \geq \frac{1}{2} |\nabla f(x)-\nabla f(y)|\, \omega^{-1}\left( \frac{1}{2M} |\nabla f(x)-\nabla f(y)| \right)
$$
for all $x, y\in \R^n$, where 
$$
M=M(\nabla f, \Rn)=\sup_{x, y\in\Rn, \, x\neq y}\frac{|\nabla f(x)-\nabla f(y)|}{\omega\left(|x-y|\right)}.
$$
In particular, if $f$ is convex then the pair $(f, \nabla f)$ satisfies $(CW^{1, \omega})$, with $\eta=1/2$, on every subset $C\subset\Rn$. 
\end{proposition}
\begin{proof}
Suppose that there exist different points $x, y\in\Rn$ such that
$$
f(x)-f(y)-\langle \nabla f(y), x-y\rangle < \frac{1}{2} |\nabla f(x)-\nabla f(y)|\, \omega^{-1}\left( \frac{1}{2M} |\nabla f(x)-\nabla f(y)| \right),
$$
and we will get a contradiction.

\noindent {\bf Case 1.} Assume further that $M=1$, $f(y)=0$, and $\nabla f(y)=0$.
By convexity this implies $f(x)\geq 0$.
Then we have 
$$
0\leq f(x)<\frac{1}{2}|\nabla f(x)|\, \omega^{-1}\left(\frac{1}{2}|\nabla f(x)|\right).
$$
Call $a=|\nabla f(x)|>0$, $b=f(x)$, set 
$$
v=-\frac{1}{|\nabla f(x)|}\nabla f(x),
$$
and define 
$$
\varphi(t)=f(x+tv)
$$
for every $t\in\R$. We have $\varphi(0)=b$, $\varphi'(0)=-a$, and $\omega$ is a modulus of continuity of the derivative $\varphi'$. This implies that 
$$
|\varphi(t)-b+at|\leq t\omega(t)
$$
for every $t\in\R^{+}$, hence also that 
$$
\varphi(t)\leq h(t) \textrm{ for all } t\in \R^{+},
$$
where $h(t)=-at+b+t\omega(t)$. By assumption, 
$$
b<\frac{1}{2} a \omega^{-1}\left(\frac{1}{2} a\right),
$$
and then Lemma \ref{h must be negative somewhere} implies that
$$
f\left( x+\omega^{-1}\left(\frac{1}{2} a\right) v\right)=\varphi\left(\omega^{-1}\left(\frac{1}{2} a\right)\right)\leq
h\left(\omega^{-1}\left(\frac{1}{2} a\right)\right)<0,
$$
which is in contradiction with the assumptions that $f$ is convex, $f(y)=0$, and $\nabla f(y)=0$. This shows that
$$
f(x)\geq \frac{1}{2}|\nabla f(x)|\, \omega^{-1}\left(\frac{1}{2}|\nabla f(x)|\right).
$$

\noindent {\bf Case 2.} Assume only that $M=1$. Define
$$
g(z)=f(z)-f(y)-\langle \nabla f(y), z-y\rangle
$$
for every $z\in\Rn$. Then $g(y)=0$ and $\nabla g(y)=0$. By Case 1, we get
$$
g(x)\geq \frac{1}{2}|\nabla g(x)|\, \omega^{-1}\left(\frac{1}{2}|\nabla g(x)|\right),
$$
and since $\nabla g(x)=\nabla f(x)-\nabla f(y)$ the Proposition is thus proved in the case when $M=1$.

\noindent {\bf Case 3.} In the general case, we may assume $M>0$ (the result is trivial for $M=0$). Consider $\psi=\frac{1}{M}f$, which satisfies the assumption of Case 2. Therefore
$$
\psi(x)-\psi(y)-\langle \nabla \psi(y), x-y\rangle \geq \frac{1}{2} |\nabla \psi(x)-\nabla \psi(y)|\, \omega^{-1}\left( \frac{1}{2} |\nabla \psi(x)-\nabla \psi(y)| \right),
$$
which is equivalent to the desired inequality.
\end{proof}

\subsection{Sufficiency} Conversely, let us now show that condition $(CW^{1, \omega})$ is sufficient in Theorem \ref{C1omega convex extension}. In the rest of this section $C$ will be an arbitrary subset of $\Rn$, and for $f:C\to\R$ and $G:C\to\R^n$ satisfying $(CW^{1, \omega})$ with $\eta=1/2$,
we will denote 
$$
m_{C}(f,G)(x)=\sup_{y\in C}\{f(y)+\langle  G(y), x-y\rangle \}
$$
for every $x\in\Rn$.

\begin{lemma}\label{existenceminimal}
Under the above assumptions $m_C(f,G)(x)$ is finite for every $x\in \Rn.$ 
\end{lemma}
\begin{proof}
Given $x\in \Rn$ we take a point $z\in C$ for which $|x-z| \leq 2 d(x,C).$ Making use of condition $(CW^{1, \omega})$ we obtain, for every $y\in C,$ that
\begin{align*}
& f(z)+ \langle G(z), x-z \rangle - \left( f(y)+ \langle G(y), x-y\rangle \right) \\
& = f(z)-f(y)- \langle G(y), z-y \rangle + \langle G(z)-G(y), x- z \rangle \\
& \geq \frac{1}{2} |G(z)-G(y)| \omega^{-1}\left( \frac{1}{2M} | G(z)- G(y)| \right) - 2|G(z)-G(y)|d(x,C).
\end{align*}
This leads us to the inequality
\begin{align*}
f(y)+ & \langle G(y), x-y\rangle   \leq f(z)+ \langle G(z), x-z \rangle \\
& + |G(z)-G(y)| \left( 2 d(x,C)- \frac{1}{2}\omega^{-1}\left( \frac{1}{2M} | G(z)- G(y)| \right)  \right).
\end{align*}
The first term in the last sum does not depend on $y.$ In the case that $\omega$ is bounded, we also have that $G$ is bounded, and this implies that the second term is also bounded by a constant only dependent on $x$ and $z$. On the other hand, if $\omega$ is unbounded, then $\omega^{-1}$ is a nonconstant convex function defined on $(0,+\infty)$, hence $\lim_{t \to +\infty} \omega^{-1}(t) = +\infty$. This implies that the second term in the above sum is bounded above by a constant only dependent on $x$ and $z$ in either case. Because $y$ is arbitrary on $C,$ we have shown that $m_C(f,G)(x)$ is finite. 
\end{proof}

\begin{lemma}\label{inequalitiapproximationminimal}
If $x\in \Rn, \: x_0,y \in C$ are such that
$$
f(x_0) + \langle G(x_0), x-x_0 \rangle \leq f(y)+ \langle G(y), x-y \rangle ,
$$
then
$|G(y)-G(x_0)| \leq 4M \omega(|x-x_0|).$
\end{lemma}
\begin{proof}
From the hyphotesis we easily have
$$
f(x_0)-f(y)- \langle G(y), x_0-y \rangle \leq \langle G(y)-G(x_0), x-x_0 \rangle \leq |G(y)-G(x_0)| |x-x_0|.
$$
Applying the inequality $(CW^{1, \omega})$ to the left-side term we see that
$$
\frac{1}{2} |G(y)-G(x_0)| \omega^{-1}\left( \frac{1}{2M} | G(y)- G(x_0)| \right) \leq |G(y)-G(x_0)| |x-x_0|,
$$
which inmediately implies
$$
| G(y)- G(x_0)| \leq 2M \omega( 2 |x-x_0| ) \leq 4M \omega(|x-x_0|).
$$
\end{proof}

As we have already mentioned, condition $(CW^{1, \omega})$ on $C$ implies condition $(W^{1,\omega})$ on $C$, which in turns implies condition $(W^{1, \omega})$ on the closure $\overline{C}$ (because Whitney's conditions for the class $C^{1, \omega}$ on any set $C$ extend uniquely to the closure of $C$), and therefore we may use Glaeser's $C^{1, \omega}$ version of the Whitney Extension Theorem in order to extend $f$ to $\R^n$ as a $C^{1, \omega}$ function. That is, we may assume that $f$ is extended to a $C^{1,\omega}(\Rn)$ function such that $\nabla f = G$ on $C$. In fact, as a consequence of the proof of the $C^{1, \omega}$ version of the Whitney Extension Theorem (see \cite{Glaeser} or \cite[Chapter (VI)]{Stein} for instance) the extension $f$ can be taken so that
\begin{equation}\label{estimation of norm in Whitney Glaeser theorem}
M(\nabla f, \Rn) := \sup_{x\neq y,\: x,y \in \Rn} \frac{|\nabla f(x)-\nabla f(y)|}{\omega(|x-y|)} \leq c(n) \max\lbrace \tilde{M} , M \rbrace.
\end{equation}
where $$ \widetilde{M}=
\sup_{x\neq y,\: x,y \in C} \left\lbrace \frac{|f(x)-f(y)-\langle G(y), x-y \rangle|}{|x-y|\omega(|x-y|)} \right\rbrace
$$
and $c(n)>0$ is a constant only depending on $n.$ In our problem, thanks to the condition $(CW^{1, \omega})$, we additionally know that $\widetilde{M} \leq M$ (see the proof of Remark \ref{CW11 implies W11}), and therefore 
\begin{equation}\label{MgradRninequality}
M(\nabla f, \Rn) := \sup_{x\neq y,\: x,y \in \Rn} \frac{|\nabla f(x)-\nabla f(y)|}{\omega(|x-y|)} \leq c(n) M.
\end{equation}

From now on we will denote
$$
m_C(f)(x):= m_C(f,\nabla f)(x)= \sup_{y\in C} \lbrace f(y)+\langle \nabla f(y),x-y \rangle \rbrace, \quad x\in \Rn.
$$
Note that according to Lemma \ref{existenceminimal} the function $m_C(f)$ is well defined on $\R^n$ and, being the supremum of a family of convex functions, is convex on $\Rn$ as well. We also see that, thanks to condition $(CW^{1, \omega})$, we have $f(x) \geq f(y)+ \langle \nabla f(y), x-y \rangle$ for every $x,y\in C,$ and hence $m_C(f)=f$ on $C.$
\begin{proposition} \label{differentiabilityminimal} For the function $m_C(f),$ the following property holds.
For every $x\in \Rn, \: x_0 \in C,$
$$
m_C(f)(x)-m_C(f)(x_0)- \langle \nabla f(x_0) , x-x_0  \rangle \leq 4 M \omega(|x-x_0|) |x-x_0|.
$$
\end{proposition}
\begin{proof}
Given $x\in \Rn \setminus C$ and $x_0 \in C,$ by definition of $m_C(f)(x),$ we can find a sequence $\lbrace y_k \rbrace_k\subset C$ such that $\lim_k (f(y_k)+ \langle \nabla f(y_k), x-y_k \rangle) = m_C(f)(x).$ We may assume that 
$$
f(y_k) + \langle \nabla f(y_k), x-y_k \rangle \geq f(x_0) +\langle \nabla f(x_0), x-x_0 \rangle
$$
for $k\geq k_0$ large enough (indeed, if $m_{C}(f)(x)=f(x_0)+\langle \nabla f(x_0), x-x_0\rangle$ we may take $y_k=x_0$ for every $k$; otherwise we have $m_{C}(f)(x)>f(x_0)+\langle \nabla f(x_0), x-x_0\rangle$ and the inequality follows by definition of $y_k$). Hence Lemma \ref{inequalitiapproximationminimal} gives 
\begin{equation} \label{inequalitygoodestimation}
|\nabla f(y_k)- \nabla f(x_0)| \leq 4M \omega(|x-x_0|) \quad \text{for} \quad k\geq k_0.
\end{equation}
On the other hand, 
\begin{align*}
0 &\leq  m_C(f)(x)- f(x_0)- \langle \nabla f (x_0), x-x_0 \rangle \\
& = \lim_k \left( f(y_k)+\langle \nabla f(y_k), x-y_k \rangle - f(x_0)-\langle \nabla f(x_0),x-x_0 \rangle \right) \\
& \leq \liminf_k \left( f(y_k)+\langle \nabla f(y_k), x-y_k \rangle - f(y_k)-\langle \nabla f(y_k) ,x_0-y_k \rangle - \langle \nabla f(x_0),x-x_0 \rangle \right) \\
& = \liminf_k \langle \nabla f(y_k)- \nabla f(x_0) ,x-x_0 \rangle \leq \liminf_k | \nabla f(y_k)- \nabla f(x_0) | |x-x_0|.
\end{align*}
By \eqref{inequalitygoodestimation}, the last term is smaller than or equal to $4M \omega(|x-x_0|)|x-x_0|.$ We thus have proved the required inequality.
\end{proof}

Now we are going to use the following result, which is implicit in \cite{Azagra} and describes the (rather rigid) global geometrical behaviour of convex functions defined on $\R^n$.

\begin{proposition}\label{rigid global behaviour of convex functions}
Let $g:\R^n\to\R$ be a convex function, and assume $g$ is not affine. Then there exist a linear function $\ell:\R^n\to\R$, a positive integer $k\leq n$, a linear subspace $X\subseteq\R^n$ of dimension $k$, and a convex function $c:X\to\R$ such that
$$
\lim_{|x|\to\infty}c(x)=\infty \, \textrm{ and } \,
g=\ell + c\circ P,
$$
where $P:\R^n\to X$ is the orthogonal projection.
\end{proposition}
\begin{proof}
We use the terminology of \cite[Section 4]{Azagra}. If $g$ is supported by an $(n+1)$-dimensional corner function then it is clear that $f=\ell+c$, with $c$ convex and $\lim_{|x|\to\infty}c(x)=\infty$, so we may take $X=\R^n$ and $P$ as the identity. Otherwise, by the proof of \cite[Lemma 4.2]{Azagra} (see also the proof of \cite[Proposition 1.6]{Azagra} on page 813 to know why \cite[Lemma 4.2]{Azagra} also holds true for nonsmooth convex functions), there exist a positive integer $k_1<n$, a linear subspace $X_1\subseteq\R^n$ of dimension $k_1$, and a convex function $c_1:X_1\to\R$ such that
$$
g=\ell_1 + c_1\circ P_1,
$$
where $P_1:\R^n\to X_1$ is the orthogonal projection and $\ell_1$ is linear. If $\lim_{|x|\to\infty}c_1(x)=\infty$ we are done. Otherwise, we apply the same argument to the convex function $c_1:X_1\to\R$ in order to obtain a subspace $X_2\subset X_1$ of dimension $k_2<k_1$, an orthogonal projection $P_2:X_1\to X_2$, a convex function $c_2:X_2\to\R$, and a linear function $\ell_2:X_1\to\R$ such that $c_1=\ell_2+c_2\circ P_2$; in particular we have
$$
g=(\ell_1 + \ell_2\circ P_1)+ c_2\circ P_2\circ P_1 :=q_2+c_2\circ Q_2,
$$
where $q_2$ is linear and $Q_2$ is still an orthogonal projection. Because $g$ is not affine, by iterating this argument at most $n-1$ times we arrive at an expression
$
g=q+c\circ P,
$
where $q$ is linear, $P:\R^n\to X$ is an orthogonal projection onto a subspace $X$ of dimension at least $1$, and $c$ is convex with $\lim_{|x|\to\infty}c(x)=\infty$.
\end{proof}

By applying the preceding Proposition to $m_{C}(f)$ we may write $m_C(f)=\ell + c \circ P$, with $\ell$ and $P$ as in the statement of the Proposition. Then, in the case that $k<n$, by taking coordinates with respect to an appropriate orthonormal basis of $\R^n$, we may assume without loss of generality that $X=\R^k\times\{0\}$ (which we identify with $\R^k$) and in particular that $P(x_1, ... , x_n)=(x_1, ... , x_k)$. Furthermore, since every linear function $\ell$ is convex, of class $C^{1,1}$, and satisfies $M(\nabla\ell, \R^n)=0$ and, besides,
\begin{equation*}
m_{C}(f)(x)-\ell(x)=\sup_{y\in C}\{f(y)-\ell(y)+\langle \nabla f(y)-\ell, x-y\rangle\}=
m_{C}(f-\ell, \nabla(f-\ell))(x),
\end{equation*}
we see that the addition or subtraction of a linear function does not affect our extension problem, and therefore we may also assume that $\ell=0$. 

Hence, from now on, we assume that $m_C(f)$ is of the form 
\begin{equation}\label{factorizationminimal}
m_C(f)=c \circ P,
\end{equation} where $c: \R^k \to \R$ is a convex function on $\R^k$ with $k\leq n$,  $\lim_{|x| \to \infty} c(x)= +\infty$, and $P : \R^n \to \R^k$ is defined by $P(x_1, \ldots , x_n ) = (x_1, \ldots, x_k)$ for all $x=(x_1, \ldots , x_n) \in \Rn.$ Of course, in the case that $k=n,$  $P$ is the identity map. In the case $k<n$, from \eqref{factorizationminimal} it is clear that $ \frac{\partial m_C(f)}{\partial x_j}(x)=0$ for every $j>k$ and $x\in C$. Recall that, by Proposition \ref{differentiabilityminimal}, the function $m_C(f)$ is differentiable at every $x\in C$, and $\nabla m_C(f) (x)= \nabla f (x).$

\begin{lemma}\label{cdifferentiablepc}
The function $c$ is $\omega$-differentiable on $P(C)$ and for each $x_0\in P(C)$ we have $\nabla c (x_0)= P( \nabla f(z_0))$ for every $z_0 \in C$ with $P(z_0)=x_0.$ In fact, we have that
$$
0\leq c(x)-c(x_0)  - \langle P( \nabla f(z_0)), x-x_0 \rangle \leq 4 M \omega \left( |x-x_0| \right)|x-x_0|.
$$
for every $x\in \R^k, \:x_0\in P(C)$, and $z_0\in C$ with $P(z_0)=x_0.$
\end{lemma}
\begin{proof}
In the case $k=n,$ our result is precisely Proposition \ref{differentiabilityminimal}. We now consider the case $k<n.$ Fix $x\in \R^k, \: x_0\in P(C)$, and $z_0 \in C$ such that $P(z_0)=x_0.$ We denote $ \overline{x_0}=Q(z_0),$ where $Q: \Rn \to \R^{n-k}$ is given by $Q(y_1,\ldots, y_n)=(y_{k+1}, \ldots , y_n)$ for every $y=(y_1,\ldots,y_n) \in \Rn.$ By equation \eqref{factorizationminimal} and the fact that $z_0=(x_0,\overline{x_0})$ we see that
\begin{align*}
c(x)& -c(x_0)  - \langle P( \nabla f(z_0)), x-x_0 \rangle \\
& = m_C(f)(x,\overline{x_0})-m_C(f)(z_0)- \langle \nabla f(z_0), (x,\overline{x_0})-z_0 \rangle.
\end{align*}
From Proposition \ref{differentiabilityminimal} we obtain that the last term is less than or equal to
$
4M \omega \left( |(x,\overline{x_0})-z_0| \right)|(x,\overline{x_0})-z_0 |= 4M \omega \left( |x-x_0| \right)|x-x_0|.
$
\end{proof}

\begin{lemma} \label{csatifiescw11}
The pair $(c,\nabla c)$ satisfies inequality $(CW^{1, \omega})$ on $P(C)$ with constant $M$ and $\eta=1/2$. That is,
$$
c(x)-c(y)- \langle \nabla c (y) , x-y \rangle \geq \frac{1}{2}| \nabla c(x) - \nabla c(y)| \omega^{-1} \left( \frac{1}{2M}  | \nabla c(x) - \nabla c(y)| \right)
$$
for every $x,y\in P(C).$ In particular, recalling Remark \ref{CW11 implies W11} we have 
$$ 0 \leq c(x)-c(y)-\langle \nabla c(y), x-y \rangle \leq 2 M |x-y|\omega(|x-y|) \quad \text{and}
$$
$$
| \nabla c(x) - \nabla c(y) | \leq 2 M \omega(|x-y|) \quad \text{for all} \quad x,y \in P(C).
$$
\end{lemma}
\begin{proof}
Given two points $x,y\in P(C),$ we find $z,w\in C$ such that $P(z)=x$ and $P(w)=y.$ Proposition \ref{cdifferentiablepc} shows that $\nabla c(x)= P(\nabla f (z) )$ and $\nabla c(y)= P( \nabla f (w) )$ and of course $c(x)=f(z), \: c(y)=f(w).$ Recall also that, in the case $k<n\:, \nabla f(z)=( P(\nabla f(z)), 0)$ and similarly for $w.$ This proves that
$$
c(x)-c(y)- \langle \nabla c (y) , x-y \rangle=f(z)-f(w)- \langle \nabla f (w) , z-w \rangle \quad \text{and} \quad
$$ 
$$
| \nabla c(x) - \nabla c(y)| = | \nabla f(z) - \nabla f(w)|.
$$
Because the pair $(f,\nabla f)$ satisfies $(CW^{1, \omega})$ on $C$ we have
$$
f(z)-f(w)- \langle \nabla f (w) , z-w \rangle \geq \frac{1}{2}| \nabla f(z) - \nabla f(w)| \omega^{-1} \left( \frac{1}{2M}  | \nabla f(z) - \nabla f(w)| \right),
$$
which inmediately implies
$$
c(x)-c(y)- \langle \nabla c (y) , x-y \rangle \geq \frac{1}{2}| \nabla c(x) - \nabla c(y)| \omega^{-1} \left( \frac{1}{2M}  | \nabla c(x) - \nabla c(y)| \right).
$$
\end{proof}

\begin{proposition}\label{csatisfiescw11closure}
The function $c$ is differentiable on $\overline{P(C)}$ and the pair $(c,\nabla c)$ satisfies inequality $(CW^{1, \omega})$ on $\overline{P(C)}$ with constant $M$ and $\eta=1/2$. In particular the pair $(c,\nabla c)$ satisfies Whitney's condition $W^{1,\omega}$ on $\overline{P(C)}$ with
$$ 0 \leq c(x)-c(y)-\langle \nabla c(y), x-y \rangle \leq 2 M |x-y|\omega(|x-y|) \quad \text{and}
$$
$$
| \nabla c(x) - \nabla c(y) | \leq 2 M \omega(|x-y|) \quad \text{for all} \quad x,y \in \overline{P(C)}.
$$
In addition,
\begin{equation}\label{comegadiffclosure}
0 \leq c(x)-c(y)-\langle \nabla c(y), x-y \rangle \leq 4 M |x-y|\omega(|x-y|) \quad x\in \R^k,\: y\in \overline{P(C)}.
\end{equation} 
\end{proposition}
\begin{proof}
By Lemma \ref{csatifiescw11} we have that $(c,\nabla c)$ satisfies $(CW^{1, \omega})$ on $P(C)$ with constant $M$. Then a routine density argument immediately shows that $\nabla c$ has a unique $\omega$-continuous extension $H$ to $\overline{P(C)}$ and that the pair $(c,\nabla c)$ also satisfies $(CW^{1, \omega})$ on $\overline{P(C)}$ with the same constant $M$. In particular, the following inequalities hold:
$$ 0 \leq c(x)-c(y)-\langle H(y), x-y \rangle \leq 2 M |x-y|\omega(|x-y|) \quad \text{and}
$$
$$
| H(x) - H(y) | \leq 2 M \omega(|x-y|) \quad \text{for all} \quad x,y \in \overline{P(C)}.
$$
Now, given $x\in \R^k$ and $y\in \overline{P(C)},$ by Lemma \ref{cdifferentiablepc} we have
$$
c(x)-c(y_p)-\langle \nabla c(y_p), x-y_p \rangle \leq 4 M |x-y_p|\omega(|x-y_p|) \quad p\in \N,
$$
for every sequence $\{ y_p \}_p $ in $P(C)$ converging to $y$. Passing to the limit as $p \to \infty$ in the above inequality and bearing in mind that $c$ and $H$ are continuous on $\overline{P(C)}$ we obtain 
$$
0 \leq c(x)-c(y)-\langle H(y), x-y \rangle \leq 4 M |x-y|\omega(|x-y|) \quad x\in \R^k,\: y\in \overline{P(C)},
$$ which in particular implies that $c$ is $\omega$-differentiable on $\overline{P(C)}$ with $\nabla c=H$.
\end{proof}

Thanks to Proposition \ref{csatisfiescw11closure} we may apply Whitney's Extension Theorem to extend $c$ from $\overline{P(C)}$ to a function $\widetilde{c} \in C^{1,\omega}(\R^k)$ such that $\nabla \widetilde{c} = \nabla c$ on $\overline{P(C)}$ and
$$
M(\nabla \widetilde{c},\R^k):= \sup_{x\neq y, \:x,y\in \R^k} \frac{|\nabla \widetilde{c} (x)- \nabla \widetilde{c}(y)|}{\omega(|x-y|)} \leq \gamma(n) M,
$$
for a constant $\gamma(n)>0$ depending only on $n$
(see equations \eqref{estimation of norm in Whitney Glaeser theorem}, \eqref{MgradRninequality}, and the argument after Lemma \ref{inequalitiapproximationminimal}).
Note that for every $x\in \R^k,$ if we pick a point $x_0\in \overline{P(C)}$ with $d(x,\overline{P(C)})=|x-x_0|,$ we obtain, by inequality \eqref{comegadiffclosure} in Proposition \ref{csatisfiescw11closure} and the facts that $\widetilde{c}(x_0)=c(x_0)$ and $\nabla \widetilde{c}(x_0)= \nabla c(x_0),$ that
\begin{equation}\label{differencesctildec}
|c(x)-\widetilde{c}(x)| \leq (4+\gamma(n)) M \omega(d(x,\overline{P(C)}))d(x,\overline{P(C)}).
\end{equation}

Our next step is constructing a function $\varphi$ of class $C^{1,\omega}(\R^k)$ which vanishes on the closed set $E:=\overline{P(C)}$ and is greater than or equal to the function $|\widetilde{c} - c|$ on $\R^k \setminus E$. To this purpose we need to use a Whitney decomposition of an open set into cubes and the corresponding partition of unity, see \cite[Chapter VI]{Stein} for an exposition of this technique. So let $\lbrace Q_j\rbrace_j$ be a decomposition of $\R^k \setminus E$ into Whitney cubes, and for a fixed number $0 <\varepsilon_0 < 1/4$ (for instance take $\varepsilon_0=1/8$) consider the corresponding cubes $\lbrace Q_j^* \rbrace_j$ with the same center as $Q_j$ and dilated by the factor $1+\varepsilon_0.$ We next sum up some of the most important properties of this cubes.

\begin{proposition} \label{propositioncubes}
The families $\lbrace Q_j \rbrace_j$ and $\lbrace Q_j^* \rbrace_j$ are sequences of compact cubes for which:
\item[(i)] $\bigcup_j Q_j = \R^k \setminus E.$
\item[(ii)] The interiors of $Q_j$ are mutually disjoint.
\item[(iii)] $\diam(Q_j) \leq d(Q_j, E) \leq 4 \diam(Q_j)$ for all $j$.
\item[(iv)] If two cubes $Q_l$ and $Q_j$ touch each other, that is $\partial Q_l \cap \partial Q_j \neq \emptyset,$ then $\diam(Q_l) \approx \diam(Q_j).$
\item[(v)] Every point of $\R^k \setminus E$ is contained in an open neighbourhood which intersects at most $N=(12)^k$ cubes of the family $\lbrace Q_j^* \rbrace_j.$ 
\item[(vi)] If $x\in Q_j, $ then $d(x, E) \leq 5 \diam(Q_j).$
\item[(vii)] If $x\in Q_j^*,$ then $ \frac{3}{4}\diam(Q_j) \leq d(x,E) \leq (6+\varepsilon_0) \diam(Q_j)$ and in particular $Q_j^* \subset \R^k \setminus E$.
\item[(viii)] If two cubes $Q_l^*$ and $Q_j^*$ are not disjoint, then $\diam(Q_l) \approx \diam(Q_j).$
\end{proposition}
\noindent Here the notation $A_j\approx B_l$ means that there exist positive constants $\gamma, \Gamma$, depending only on the dimension $k$, such that $\gamma A_j\leq B_l\leq \Gamma A_j$ for all $j, l$ satisfying the properties specified in each case.
The following Proposition summarizes the basic properties of the Whitney partition of unity associated to these cubes.
\begin{proposition} \label{propositionpartition}
There exists a sequence of functions $\lbrace \varphi_j \rbrace_j$ defined on $\R^k \setminus E$ such that
\begin{itemize}
\item[(i)] $\varphi_j \in C^\infty(\R^k \setminus E).$ 
\item[(ii)] $0\leq \varphi_j \leq 1$ on $\R^k \setminus E$ and $\sop (\varphi_j) \subseteq Q_j^*.$ 
\item[(iii)] $\sum_j \varphi_j =1$ on $\R^k \setminus E$. 
\item[(vi)] For every multiindex $\alpha$ there exists a constant $A_\alpha>0$, depending only on $\alpha $ and on the dimension $k$, such that
$$
| D^\alpha \varphi_j (x) | \leq A_\alpha \diam(Q_j)^{-|\alpha|},
$$
for all $x\in \R^k \setminus E$ and for all $j$. 
\end{itemize} 
\end{proposition}

The statements contained in the following Proposition must look fairly obvious to those readers well acquainted with Whitney's techniques, but we have not been able to find an explicit reference, so we include a proof for the general reader's convenience.

\begin{proposition}\label{c1wregularization}
Consider the family of cubes $\lbrace Q_j \rbrace_j$ asociated to $\R^k \setminus E$ and its partition of unity $\lbrace \varphi_j \rbrace_j$ as in the preceding Propositions. Suppose that there is a sequence of nonnegative numbers $\lbrace p_j \rbrace_j$ and a positive constant $\lambda>0$ such that $p_j \leq \lambda \: \omega ( \diam(Q_j) ) \diam(Q_j),$ for every $j$. Then, the function defined by
$$
 \varphi(x)   =  \left\lbrace
	\begin{array}{ccl}
	\sum_j p_j \: \varphi_j(x) & \mbox{if } & x\in \R^k \setminus E, \\
	0  & \mbox{if }&  x\in E
	\end{array}
	\right.
$$
is of class $C^{1,\omega}(\R^k)$ and there is a constant $\gamma (k)>0$, depending only on the dimension $k$, such that
\begin{equation} \label{gammaninequality}
M(\nabla \varphi , \R^k ):= \sup_{x,y\in \R^k,\: x \neq y} \frac{|\nabla \varphi (x)- \nabla \varphi (y)|}{\omega\left(|x-y|\right)} \leq \gamma (k) \lambda.
\end{equation}
\end{proposition}
\begin{proof}
Let us start with the proof of (i). Since every point in $\R^k \setminus E$ has an open neighbourhood which intersects at most $N=(12)^k$ cubes, and all the functions $\varphi_j$ are of class $C^\infty,$ it is clear that $\varphi \in C^\infty(\R^k \setminus E).$ Given a point $x\in \R^k$, by our assumptions on the sequence $\lbrace p_j \rbrace_j$ we easily have
$$
\varphi(x)= \sum_{Q_j^* \ni x} p_j \varphi_j(x) \leq \lambda \sum_{Q_j^* \ni x} \omega ( \diam(Q_j) ) \diam(Q_j) \varphi_j(x).
$$
Using Proposition \ref{propositioncubes} we have that $\diam(Q_j) \leq \frac{4}{3} d(x,E)$ for those $j$ such that $x\in Q_j^*.$ Then, we can write
\begin{align*}
\varphi(x) & \leq \lambda \sum_{Q_j^* \ni x} \omega \left( \frac{4}{3} d(x,E) \right) \frac{4}{3} d(x,E) \varphi_j(x) \leq \\
&   \lambda \left( \frac{4}{3} \right)^2 \omega ( d(x,E)) d(x,E) \sum_{Q_j^* \ni x} \varphi_j(x) = \left( \frac{4}{3} \right)^2 \lambda \omega ( d(x,E)) d(x,E).
\end{align*}
In particular, due to the fact that $\omega(0^+)=0,$ the above estimation shows that $\varphi$ is differentiable on $E$, with $\nabla \varphi = 0$ on $C.$ We next give an estimation for $\nabla \varphi$. Bearing in mind Proposition \ref{propositionpartition}, we set 
\begin{equation} \label{definitionAn}
A= A(k):= \max \lbrace \sqrt{k} A_\alpha, \: |\alpha| \leq 2 \rbrace.
\end{equation}

Given a point $x\in \R^k$, we have that
\begin{align*} 
|\nabla \varphi (x) | & \leq \sum_{Q_j^* \ni x} p_j  | \nabla\varphi_j (x) | \leq A \sum_{Q_j^* \ni x} p_j \diam(Q_j)^{-1} \leq  A  \lambda \sum_{Q_j^* \ni x} \omega ( \diam(Q_j) )\\
&\leq  A \: \lambda  \sum_{Q_j^* \ni x} \omega \left( \frac{4}{3} d(x,E) \right) \leq \frac{4 A N }{3} \lambda \: \omega ( d(x,E)).
\end{align*}
Summing up,
\begin{equation} \label{boundgradientphi}
|\nabla \varphi (x) | \leq \frac{4 A N }{3} \lambda \: \omega ( d(x,E)) \quad \text{for every} \quad x\in \R^k.
\end{equation}
Note that the above shows inequality \eqref{gammaninequality} when $x\in \R^k$ and $y\in E$. Hence, in the rest of the proof we only have consider the situation where $x,y$ are such that $x\neq y$ and $x,y\in \R^k \setminus E$. However, we will still have to separately consider two cases. Let us denote $L:=[x,y],$ the line segment connecting $x$ to $y.$ 

\noindent {\bf Case 1:} $d(L,E) \geq |x-y|.$ Take a multi-index $\alpha$ with $|\alpha|=1.$ Because in this case the segment $L$ is necessarily contained in $\R^k \setminus E$, and the function $\varphi$ is of class $C^2$ on $\R^k\setminus E$, we can write
\begin{equation}\label{differencegradients}
| D^\alpha \varphi (x)- D^\alpha \varphi(y) | \leq \left( \sup_{z\in L} | \nabla( D^\alpha \varphi )(z) | \right) |x-y|.
\end{equation}
Using Proposition \ref{propositionpartition} and the definition of $A$ in \eqref{definitionAn}, we may write
$$
| \nabla( D^\alpha \varphi )(z) |  \leq \sum_{Q_j^* \ni z} p_j | \nabla( D^\alpha \varphi_j )(z) | \leq A \sum_{Q_j^* \ni z} p_j \diam(Q_j)^{-2} \leq 
 A \lambda \sum_{Q_j^* \ni z} \frac{\omega(\diam(Q_j))}{\diam(Q_j)}.
 $$
By Proposition \ref{propositioncubes}, we have that $(6+\varepsilon_0) \diam(Q_j) \geq d(z,E) \geq d(L,E) \geq |x-y|$ for those $j$ with $Q_j^* \ni z$, and by the properties of $\omega,$ we have that
$$
A \lambda \sum_{Q_j^* \ni z} \frac{\omega(\diam(Q_j))}{\diam(Q_j)} \leq A \lambda \sum_{Q_j^* \ni z} \frac{\omega \left( \frac{|x-y|}{6+\varepsilon_0} \right)}{\frac{|x-y|}{6+\varepsilon_0}} \leq (6+\varepsilon_0) A N \lambda \frac{\omega(|x-y|)}{|x-y|}.
$$
Therefore, by substituting in \eqref{differencegradients} we find  that
$$
| \nabla \varphi (x)- \nabla \varphi (y) | \leq (6+\varepsilon_0) \sqrt{k} A N \lambda \omega(|x-y|).
$$

\noindent {\bf Case 2:} $d(L,E) \leq |x-y|.$ Take points $x'\in L$ and $y'\in E$ such that $d(L,E) = |x'-y'| \leq |x-y|.$ We have that
$$
|x-y'| \leq |x-x'| +|y'-x'| \leq |x-y| + |x'-y'| \leq 2 |x-y|,
$$
and similarly we obtain $|y-y'| \leq 2 |x-y|.$ Hence, if we use \eqref{boundgradientphi} we obtain
\begin{align*}
&| \nabla \varphi (x) - \nabla \varphi (y) |  \leq | \nabla \varphi (x) | + | \nabla \varphi (y) | \leq \frac{4 A N }{3} \lambda \left( \omega(d(x,E)) + \omega (d(y,E)) \right)  \\ &  \leq \frac{4 A N }{3} \lambda \left( \omega(|x-y'|) + \omega (|y-y'|) \right) \leq \frac{8 A N }{3} \lambda \omega( 2|x-y| ) \leq \frac{16 A N }{3} \lambda \omega ( |x-y| ).
\end{align*}

If we call
$$
\gamma(k)=\max\left\lbrace \frac{4 A N }{3},(6+\varepsilon_0) \sqrt{k} A N, \frac{16 A N }{3} \right\rbrace = (6+\varepsilon_0) \sqrt{k} A N,
$$
we get \eqref{gammaninequality}. 
\end{proof}

Let us continue with the proof of Theorem \ref{C1omega convex extension}. We consider a decomposition of $\R^k \setminus \overline{P(C)}$ into Whitney's cubes $\{ Q_j, Q_j^* \}_j$ and its asociated partition of unity $\{ \varphi_j \}_j$ (see Propositions \ref{propositioncubes} and \ref{propositionpartition}). If we define
$$
p_j:= \sup_{x\in Q_j^*} | c(x)-\widetilde{c}(x) | \quad \text{for all} \quad j,
$$
we have, thanks to \eqref{differencesctildec}, that $p_j \lesssim M \omega(\diam(Q_j))\diam(Q_j)$ for all $j$. Then, according to the preceding Proposition, the function
$$
 \varphi(x)  : =  \left\lbrace
	\begin{array}{ccl}
	\sum_j p_j \: \varphi_j(x) & \mbox{if } & x\in \R^k \setminus \overline{P(C)}, \\
	0  & \mbox{if }&  x\in \overline{P(C)}
	\end{array}
	\right.
$$
is of class $C^{1,\omega}(\R^k)$ with $\varphi=0$ and $\nabla \varphi=0$ on $\overline{P(C)}$ and $M(\nabla \varphi, \R^k) \lesssim M.$ In addition, by the definition of $\{ p_j \}_j$ we easily see that $\varphi \geq |c-\widetilde{c}|$ on $\R^k.$ Now we set $\psi:= \widetilde{c} + \varphi$ on $\R^k,$ which is of class $C^{1,\omega}(\R^k)$ with $\psi=\widetilde{c}=c$ and $\nabla \psi= \nabla \widetilde{c}= \nabla c$ on $\overline{P(C)}$, and also $M(\nabla \psi, \R^k) \lesssim M.$ On the other hand, it is clear that $\psi \geq c$ on $\R^k$ and this in particular implies that $\lim_{|x| \to \infty} \psi(x) = +\infty.$ 

Now we will use the following.

\begin{theorem}[Kirchheim-Kristensen]\label{KKtheorem}
If $H : \R^k \to \R$ is an $\omega$-differentiable function on $\R^k$ such that $\lim_{|x| \to +\infty} H (x) = +\infty,$ then the convex envelope $\textrm{conv}(H)$ of $H$ is a convex function of class $C^{1,\omega}(\R^k)$, and
$$
M( \nabla \textrm{conv}(H), \R^k) \leq 4 (n+1) M(\nabla H, \R^k).
$$
\end{theorem}
\noindent As a matter of fact, Theorem \ref{KKtheorem} is a restatement of a particular case of Kirchheim and Kristensen's result in \cite{KirchheimKristensen}, and the estimation on the modulus of continuity of the gradient of the convex envelope does not appear in their original statement, but it does follow, with some easy extra work, from their proof.

Since $\psi \in C^{1,\omega}(\R^k)$ with $M( \nabla \psi , \R^k) \lesssim M$ and $\lim_{|x|\to \infty} \psi(x)=+\infty,$ if we define $\widetilde{F} := \textrm{conv}(\psi),$ Theorem \ref{KKtheorem} implies that our function $\widetilde{F}$ is a convex function of class $C^{1,\omega}(\R^k)$ and satisfies $M( \nabla \widetilde{F}, \R^k ) \lesssim M.$ Since $\psi \geq c$ and $c$ is convex, we have $\widetilde{F} \geq c$ on $\R^k.$ On the other hand, because $\psi=c$ on $\overline{P(C)},$ we must have $\widetilde{F} \leq c$ on $\overline{P(C)}$ and therefore $\widetilde{F}=c$ on $\overline{P(C)}.$
Moreover, because $c \leq \widetilde{F}, \: c$ is convex and $c= \widetilde{F}$ on $\overline{P(C)},$ Lemma \ref{differentiability criterion for convex functions} implies that $\nabla \widetilde{F} = \nabla c$ on $\overline{P(C)}.$ 

Finally we define $F:= \widetilde{F} \circ P$ on $\Rn.$ We immediately see that $F$ is convex and of class $C^{1,\omega}$ on $\R^n$. Thanks to equation \eqref{factorizationminimal} and the fact that $\widetilde{F}= c$ and $\nabla \widetilde{F}= \nabla c$ on $P(C)$ we have for all $x\in C,$
$$
F(x)= \widetilde{F}(P(x)) = c (P(x))= m_C(f)(x)=f(x) \quad \text{and}
$$
$$
\nabla F(x) = \nabla \tilde{F} (P(x)) \circ D P = \nabla c(P(x)) \circ D P = \nabla m_C(f)(x) = \nabla f(x)=G(x).
$$
Since $M( \nabla \widetilde{F}, \R^k ) \lesssim M,$ for every $x,y\in \R^n,$ the gradient of $F$ satisfies
$$
|\nabla F(x)- \nabla F(y) | = |\nabla \widetilde{F}(P(x))- \nabla \widetilde{F}(P(y))| \lesssim M \omega(|P(x)-P(y)|)\lesssim M \omega(|x-y|).
$$
This allows us to conclude the existence of a constant $\gamma(n)>0$ depending only on $n$ such that
$$
M(\nabla F, \R^n) := \sup_{x\neq y,\: x,y\in\R^n} \frac{|\nabla F(x)-\nabla F(y)|}{\omega(|x-y|)} \leq \gamma(n) M.
$$
The proof of Theorem \ref{C1omega convex extension} is complete.


\medskip

\subsection{Proof of Theorem \ref{corollary for C11 convex bodies}}
\noindent $(2)\implies (1)$: Set $C=K\cup\{0\}$. We may assume $\eta>0$ small enough, and choose $\alpha>0$ sufficiently close to $1$, so that
$$
0<1-\alpha +\frac{\eta}{2M} <\min_{y\in K}\langle N(y), y\rangle
$$
(this is possible thanks to condition $(\mathcal{O})$, $M$-Lipschitzness of $N$ and compactness of $K$; notice in particular that $0\notin K$).
Now define $f:C\to \R$ and $G:C\to\R^n$ by
$$
 f(y)   =  \left\lbrace
	\begin{array}{ccl}
	1 & \mbox{ if } y\in K \\
	\alpha  & \mbox{ if }  y=0,
	\end{array}
	\right. \textrm{ and }
 G(y)   =  \left\lbrace
	\begin{array}{ccl}
	N(y) & \mbox{ if } y\in K \\
	0  & \mbox{ if }  y= 0.
	\end{array}
	\right.
$$
By using condition $(\mathcal{K}\mathcal{W}^{1,1})$,
it is straightforward to check that $(f,G)$ satisfies condition $(CW^{1,1})$. Therefore, according to Theorem \ref{C11 convex extension}, there exists a convex function $F\in C^{1,1}(\R^n)$ such that $F=f$ and $\nabla F=G$ on $C$, and the proof of Theorem \ref{C11 convex extension} indicates that $F$ can be taken so as to satisfy $\lim_{|x|\to\infty}F(x)=\infty$. We now define
$V=\{x\in\R^n \, : \, F(x)\leq 1\}$. The rest of the proof is similar to that of Theorem \ref{corollary for C1 convex bodies}, and we leave it to the reader's care. $\Box$

\section*{Acknowledgement}
The authors wish to thank the referee for several suggestions that improved the exposition, and for the statement of Theorem \ref{second quantitative version of C11 convex extension}.


\end{document}